\documentclass[11pt]{amsart}

\usepackage{epsfig, color}

\usepackage{amsmath,amssymb}
\usepackage{amsthm,tikz}
\usepackage{hyperref}
\usepackage[margin=1.1in]{geometry}

\numberwithin{equation}{section}

\newtheorem{prop}{Proposition}
\newtheorem{lemma}[prop]{Lemma}

\newtheorem{thm}[prop]{Theorem}
\newtheorem{cor}[prop]{Corollary}

\numberwithin{prop}{section}

\theoremstyle{definition}
\newtheorem{defn}[prop]{Definition}

\newtheorem{rmk}[prop]{Remark}

\let\waymore\gg

\newcommand{\del}{\partial}
\newcommand{\delb}{\bar{\partial}}\newcommand{\dt}{\frac{\partial}{\partial t}}
\newcommand{\brs}[1]{\left| #1 \right|}

\newcommand{\gG}{\Gamma}
\renewcommand{\gg}{\gamma}
\newcommand{\gD}{\Delta}
\newcommand{\gd}{\delta}
\newcommand{\gs}{\sigma}

\newcommand{\gl}{\lambda}

\newcommand{\gw}{\omega}
\newcommand{\ga}{\alpha}
\newcommand{\gb}{\beta}
\newcommand{\gL}{\Lambda}

\newcommand{\N}{\nabla}
\newcommand{\FF}{\mathcal F}

\newcommand{\VV}{\mathcal V}

\newcommand{\til}[1]{\widetilde{#1}}

\renewcommand{\bar}[1]{\overline{#1}}

\renewcommand{\i}{\sqrt{-1}}

\newcommand{\hook}{\lrcorner}

\newcommand{\bga}{\bar{\ga}}

\newcommand{\IP}[1]{\left<#1\right>}

\newcommand{\HH}{\mathcal{H}}

\DeclareMathOperator{\Rc}{Rc}
\DeclareMathOperator{\Rm}{Rm}

\DeclareMathOperator{\tr}{tr}

\DeclareMathOperator{\Id}{Id}

\DeclareMathOperator{\grad}{grad}
\DeclareMathOperator{\Vol}{Vol}

\DeclareMathOperator{\Kod}{Kod}
\DeclareMathOperator{\End}{End}
\DeclareMathOperator{\Real}{Re}

\begin{document}

\title[Classification of solitons for pluriclosed flow on complex 
surfaces]{Classification of solitons for pluriclosed flow on complex surfaces}

\begin{abstract} We give a classification of compact solitons for the 
pluriclosed flow on complex surfaces.  First, by exploiting results from the 
Kodaira classification of surfaces, we show that the complex surface underlying 
a  
soliton must be K\"ahler except for the possibility of steady solitons on 
minimal Hopf surfaces.  Then, we 
construct steady solitons on all class $1$ Hopf surfaces by exploiting a 
natural symmetry ansatz.
\end{abstract}

\date{\today}

\author{Jeffrey Streets}
\address{Rowland Hall\\
         University of California\\
         Irvine, CA 92617}
\email{\href{mailto:jstreets@uci.edu}{jstreets@uci.edu}}

\maketitle

\section{Introduction}

The pluriclosed flow is a geometric evolution equation generalizing the 
K\"ahler-Ricci flow to more general complex, non-K\"ahler manifolds.  As shown 
in (\cite{PCFReg}, Theorem 1.1), this flow is gauge equivalent to a certain 
natural coupling of the Ricci flow and the heat flow for a 
closed three-form first appearing in mathematical physics, namely
\begin{gather} \label{f:GRF}
\begin{split}
\dt g_{ij} =&\ - 2 \Rc_{ij} + \tfrac{1}{2} H_i^{pq} H_{jpq},\\
\dt H =&\ \gD_d H.
\end{split}
\end{gather}
We will refer to this system of equations informally as ``generalized Ricci 
flow."  As shown in (\cite{OSW} Proposition 3.1),  generalized Ricci flow 
admits a 
Perelman-type 
energy monotonicity formula, and is in fact the gradient flow of a certain 
Schr\"odinger operator.  This indicates that, as in the case of Ricci flow, 
soliton solutions of (\ref{f:GRF}) should be expected as long time limits and 
singularity models for this flow.  The steady soliton equations implied by the 
gradient formulation take the form
\begin{gather}
\begin{split}
\Rc - \tfrac{1}{4} H^2 + \N^2 f =&\ 0,\\
\tfrac{1}{2} d^* H + i_{\tfrac{1}{2} \N f} H =&\ 0.
\end{split}
\end{gather}
  Thus a fundamental step in understanding the regularity and long time 
behavior the 
pluriclosed flow is to understand the existence and uniqueness of solutions to 
this system of equations.

The first step is to understand fixed points of the pluriclosed flow, where $f$ 
above is a constant function.  On the diagonal Hopf surfaces, there is a 
well-known metric which is 
a non-K\"ahler fixed point of pluriclosed flow, which is in fact the 
\emph{unique} non-K\"ahler fixed point 
on compact complex surfaces up to scaling (cf. \S \ref{s:inducedsurfaces}).  To address 
genuine, non-trivial soliton solutions, first recall the fundamental fact that 
any compact steady soliton for the Ricci flow is 
automatically Einstein (\cite{Ivey3} Proposition 1, \cite{Perelman1} \S 2.4).  
However, the Bianchi type 
identities/monotonicity formulae behind these proofs do not immediately 
generalize to the case of 
pluriclosed flow, and thus one is left with the possibility that nontrivial 
compact steady solitons may exist.  Moreover, natural conjectures on the 
pluriclosed flow loosely suggest the existence of such objects.  In particular, 
the main existence conjecture for pluriclosed flow (\cite{PCFReg} Conjecture 
5.2) 
suggests that, on minimal Hopf surfaces, the flow exists on $[0,\infty)$ and is 
nonsingular.  If true, a standard argument using the Perelman $\FF$-functional 
referenced above would then imply convergence to a 
compact steady soliton on such surfaces.

The main result of this paper confirms this expectation, and gives a nearly 
complete classification of compact pluriclosed 
solitons on complex surfaces.  

\begin{thm} \label{t:solitonthm} Let $(M^4, J)$ be a compact complex surface.  
Suppose $(g, f)$ is a pluriclosed soliton on $(M, J)$.
\begin{enumerate}
 \item If $(g,f)$ is an expanding soliton, then $(M, J)$ is K\"ahler, $f \equiv 
\mbox{const}$ and $g$ is the Aubin-Yau metric.
 \item If $(g,f)$ is a shrinking soliton, then $(M, J)$ is K\"ahler and $(g,f)$ 
is a K\"ahler-Ricci shrinking soliton.
 \item If $(g,f)$ is a steady soliton, then either
\begin{enumerate}
\item $(M, J)$ is K\"ahler, $f \equiv \mbox{const}$ and $g$ is a 
Calabi-Yau metric, or
\item $(M, J)$ is biholomorphic to a minimal Hopf 
surface.
\end{enumerate}
\end{enumerate}
Furthermore, on minimal Hopf surfaces of class 1, there exists a nontrivial 
pluriclosed 
steady soliton.
\end{thm}

\begin{rmk}
\begin{enumerate}
\item The 
rigidity of expanding solitons follows from a Perelman-type expanding entropy 
formula for generalized Ricci flow, as already observed in (\cite{PCFReg} 
Corollary 6.11).  In fact this rigidity holds in any dimension, and moreover 
any expanding soliton of generalized Ricci flow (i.e. not necessarily 
pluriclosed) must satisfy $H \equiv 0$ with the underlying metric Einstein 
(\cite{Streetsexpent} Proposition 6.4).
\item In the shrinking case there is a 
complete picture of the existence and uniqueness following from a long series 
of 
works in K\"ahler geometry.  The complex surface underlying a K\"ahler-Ricci 
soliton must be Fano, and so for complex surfaces must be either $\mathbb 
{CP}^1 \times \mathbb{CP}^1$ or $\mathbb {CP}^2 \# k \bar{\mathbb{CP}}^2, 1 
\leq k \leq 8$.  The spaces $\mathbb{CP}^1 \times \mathbb {CP}^1$ and $\mathbb 
{CP}^2$ admit natural K\"ahler-Einstein metrics coming from their realization 
as symmetric spaces.  The existence of K\"ahler-Einstein metrics on blowups of 
$\mathbb {CP}^2$ for the cases $k=3,4$ was established by Tian-Yau 
\cite{TianYauKE}, and for $5 \leq k \leq 8$ by Tian \cite{TianKE}.  The cases 
$k=1,2$ cannot admit K\"ahler-Einstein metrics due to Matsushima's obstruction 
\cite{Matsushima}, but do admit K\"ahler-Ricci solitons.  By exploiting a 
codimension $1$ 
symmetry reduction, Koiso \cite{Koisorot} established the existence of a 
soliton in 
the case $k=1$.   Exploiting toric symmetry, Wang-Zhu \cite{WangZhuKRS} constructed a soliton in the case $k=2$.  Uniqueness of these solitons was established by Tian-Zhu 
\cite{TianZhu}.
\item The construction of steady solitons on Hopf surfaces of class $1$ leaves 
open the question of existence on Hopf surfaces of class $0$.  While the loose 
argument given above suggesting the existence of steady solitons applies in 
principle to these surfaces, it is possible that a ``jumping'' phenomenon 
occurs for pluriclosed flow, where the convergence in Cheeger-Gromov topology 
allows for the biholomorphism type of the complex structure to change in the 
limit.  As it is known that Hopf surfaces of class $1$ occur as the central 
fiber in a family of complex surfaces, all other fibers of which all 
biholomorphic to the same class $0$ surface, this possibility could certainly 
occur.  Furthermore, we do not address the question of uniqueness.  This question is likely difficult, since for 
instance the proof of uniqueness of K\"ahler-Ricci solitons 
(\cite{ZhuUniqueness, TianZhu}) requires a number of delicate a priori 
estimates which are difficult to generalize to this setting.
\end{enumerate}
\end{rmk}

The proof of Theorem \ref{t:solitonthm} breaks into two phases.  First we 
establish the rigidity of shrinking solitons, as well as the restriction of the 
possible biholomorphism types of steady solitons.  These require exhibiting a 
priori structural results for 
solitons, and then comparing this structure against results from the 
Kodaira classification 
of surfaces.  An important initial step is to identify a natural holomorphic 
vector field associated to a steady pluriclosed soliton.  As is well-known, for 
K\"ahler-Ricci solitons, the gradient vector field associated to the soliton 
function is holomorphic.  This follows directly from the soliton equation for 
the Riemannian metric together with pointwise identities exploiting the 
K\"ahler condition (see Proposition \ref{p:KRSholo}).  Due to the weaker 
integrability of pluriclosed metrics, we cannot expect the same behavior for 
pluriclosed solitons.  Rather, as we show in \S \ref{s:solitonclass}, the 
vector field $\theta^{\sharp} + \N f$, where $\theta^{\sharp}$ is the Lee 
vector field (see \S \ref{s:background}), will always be a nontrivial 
holomorphic 
vector field, unless the metric is already K\"ahler-Einstein.  Another piece of 
input is a B\"ochner argument showing that pluriclosed solitons on K\"ahler 
manifolds are automatically K\"ahler-Ricci solitons.  With these tools in hand, 
we exploit deep results from the Kodaira classification of surfaces to 
rule out backgrounds other than minimal Hopf surfaces.

The second phase of the proof is a nearly explicit construction of solitons on 
class $1$ Hopf surfaces.  To begin we use the correspondence between Sasakian 
$3$-manifolds and Hermitian 
surfaces first introduced by Vaisman \cite{VaismanLCK}.  This in particular 
allows us to construct pluriclosed metrics with a pair of holomorphic Killing 
fields.  We note here that this construction was originally used to construct 
locally conformally K\"ahler metrics on complex surfaces \cite{Belgunmetric, 
GauduchonLCK}, and while it is 
natural to imagine that such metrics play a distinguished role for pluriclosed 
flow, this does not seem to be the case (cf. Remark \ref{r:notSLCK}).  
Nonetheless, the 
pluriclosed flow will preserve invariance by holomorphic vector fields, and 
thus it is natural to look for solitons in this ansatz.  The next step is to 
reduce the flow and soliton equations in this setting to equations depending 
only on the directions transverse to the foliation generated by the holomorphic 
Killing fields.\footnote{Surprisingly, these flow equations 
reduce to a natural coupling of Ricci flow and Yang-Mills flow 
(Ricci-Yang-Mills flow) introduced by the author \cite{Streetsthesis} and Young 
\cite{Youngthesis} (see Remark \ref{r:RYM}).}

We complete the construction by constructing solutions in this symmetry class.  
Note that, as described so far, the construction is of cohomogeneity $2$, and 
thus one would expect to be faced with solving a PDE.  However, a crucial extra 
symmetry arises in this setting which allows for a further reduction.  It is by 
now a well-known fact (\cite{Ricciflowbook} p. 241, \cite{ChenLuTian}) that the 
shrinking Ricci soliton equations on 
Riemann surfaces acquire a natural holomorphic isometry generated by $J \N f$, 
where $f$ denotes the soliton function.  This feature persists to our setting 
(Proposition \ref{p:invprop}), which allows us to reduce to an ordinary 
differential equation in a single parameter generated by $\N f$.  In the case 
of elliptic Hopf surfaces, the isometry $J \N f$ corresponds to the rotational 
symmetry of the base orbifold, which is either a so-called ``teardrop'' or 
``football'' (see \S \ref{s:inducedsurfaces}).  In the general, irregular, 
case, the vector 
field $J \N f$ is a certain explicit holomorphic vector field which is 
transverse to the underlying Sasaki structure.  Hence the problem is now 
reduced to one parameter, and so we are faced with solving a certain nonlinear 
system of ODE.  Through a series of estimates we produce the necessary 
solutions as well as a  qualitative picture of their behavior, finishing the 
existence proof.

\begin{rmk} Our construction is closely related to some classic constructions 
for 
Ricci flow.  In his study of the Ricci flow on surfaces Hamilton 
\cite{Hamiltonsurfaces} investigated Ricci solitons on surfaces, showing that 
there are no compact shrinking solitons other than the round metric on $S^2$.  
He also mentions nontrivial solitons which exist on orbifolds.  A further 
analysis of these resulting ODEs yielded the existence of Sasaki-Ricci solitons 
on $3$-manifolds in \cite{WangZhangSRF3}.   These connections are however 
thematic, and our analysis is distinct from that underlying Ricci or 
Sasaki-Ricci solitons.
\end{rmk}

We also note that Theorem \ref{t:solitonthm} provides an interesting conceptual 
distinction between generalized Ricci flow and Ricci flow.  Theorem 
\ref{t:solitonthm} provides examples of nontrivial 
gradient steady solitons for the generalized Ricci flow in dimension $4$, and 
by 
taking products with tori yields such structures in all dimensions $n \geq 4$.  
Interestingly, a careful examination of the our construction reveals a product 
structure, yielding an example of a three-dimensional soliton as well.  Thus we 
obtain:

\begin{cor} \label{c:solitoncor} There exist nontrivial compact generalized 
steady 
solitons in all dimensions $n \geq 3$.
\end{cor}

Here is an outline of the paper.  We begin in \S \ref{s:background} with a 
discussion of relevant background material.  In \S \ref{s:solitonclass} we 
prove the first 
part of Theorem \ref{t:solitonthm}, restricting the class of complex surfaces 
which can admit solitons to minimal Hopf surfaces.  In \S 
\ref{s:SasakiHermitian} we review the basic correspondence between Sasakian 
$3$-manifolds and associated Hermitian manifolds.  Next in \S \ref{s:invgeom} 
we make a more general investigation of pluriclosed metrics on complex surfaces 
admitting a pair of holomorphic Killing fields.  We reduce the pluriclosed flow 
and soliton equations in this invariant setting in \S \ref{s:invPCF}.  Lastly, 
in \S \ref{s:existence} we finish the existence proof through a detailed ODE 
analysis.

\vskip 0.1in

\textbf{Acknowledgements:} The author was supported by the NSF via DMS-1454854. 
 The author would like to thank Daniel Agress, Florin Belgun, Tristan Collins, 
Paul Gauduchon, and 
Massimiliano Pontecorvo for helpful conversations.

\section{Background} \label{s:background}

In this section we provide some very brief background on pluriclosed flow, 
referring the reader to \cite{PCF,PCFReg} for more details.  First, given a 
complex manifold $(M^{2n}, J)$, a Hermitian metric $g$ is called pluriclosed if 
its associated K\"ahler form $\gw$ satisfies $\i\del\delb \gw = 0$.  Associated 
to a pluriclosed metric is the Bismut connection, defined by
\begin{align*}
\N^B = \N^{LC} + \tfrac{1}{2} g^{-1} d^c \gw,
\end{align*}
where $\N^{LC}$ denotes the Levi-Civita connection and $d^c = \i \left(\delb - 
\del \right)$.  This is the unique Hermitian connection whose torsion is skew 
symmetric, and in this case also closed since $\gw$ is pluriclosed.  Let 
$\Omega^B$ denote the curvature of this connection, and let
\begin{align*}
\rho_B(X,Y) = \tfrac{1}{2} \Omega(X,Y,e_i,J e_i)
\end{align*}
denote the Bismut-Ricci form.  By general theory $d \rho_B = 0$, and $\rho_B 
\in \pi c_1$.  However, unlike the Chern-Ricci form, $\rho_B \notin 
\Lambda^{1,1}$, and we let $\rho_B^{1,1}$ denote the $(1,1)$ piece.  The 
pluriclosed flow equation takes the form
\begin{align*}
\dt \gw =&\ - \rho_B^{1,1}.
\end{align*}
As shown in (\cite{PCF} Theorem 1.2), this is a parabolic equation for $\gw$ 
which admits short time solutions on compact manifolds.

Connecting the pluriclosed flow to the system (\ref{f:GRF}) requires a 
nontrivial gauge transformation.  Associated to a Hermitian metric we have the 
Lee form $\theta = d^* \gw \circ J$.  Let $\gw_t$ denote a solution to 
pluriclosed flow, let $g_t$ denote the one parameter family of associated 
metrics, and $H_t = d^c \gw_t$ the one parameter family of associated torsion 
tensors.  Let $\phi_t$ denote the one-parameter family of diffeomorphisms 
generated by $\theta_t^{\sharp}$.  Then (\cite{PCFReg} Theorem 6.5) yields that 
$(\phi_t^* g_t, \phi_t^* H_t)$ is a solution to (\ref{f:GRF}), up to 
reparameterizing time.  We record one key curvature identity behind this 
theorem which is relevant to what follows:

\begin{prop} (cf. \cite{PCFReg} Theorem 6.5) \label{p:BismutGR} Let $(M^{2n}, 
g, J)$ be a complex manifold with pluriclosed metric $g$.  Then
\begin{align*}
\rho^{1,1}_B(J \cdot, \cdot) = \Rc - \tfrac{1}{4} H^2 + \tfrac{1}{2} 
L_{\theta^{\sharp}} g.
\end{align*}
\end{prop}

As a final important introductory remark, we note that there is a 
classification of fixed points of pluriclosed flow on complex surfaces.  As it 
happens there is only one non-K\"ahler example which we now describe.  Define a 
metric on $\mathbb C^2 \backslash \{0,0\}$ via
\begin{align} \label{f:hopfmetric}
 g_{\mbox{\tiny{Hopf}}} = \frac{g_{E}}{\brs{z_1}^2 + \brs{z_2}^2},
\end{align}
where $g_E$ denotes the standard Euclidean metric on $\mathbb C^2$.  Note that 
the metric 
$g_{\mbox{\tiny{Hopf}}}$ is isometric to the metric $dr^2 \oplus g_{S^3}$ on 
$\mathbb R \times S^3$, with the $\mathbb R$ factor spanned by the radial 
direction.  Furthermore, $g_{\mbox{\tiny{Hopf}}}$ is compatible with the 
standard complex structure, and direct calculations show that it is 
pluriclosed, and moreover $\rho_B^{1,1}(g_{\mbox{\tiny{Hopf}}}) = 0$.  It is 
also invariant under actions $(z_1, z_2) \to (\ga z_1, \gb z_2)$, $\brs{\ga} = 
\brs{\gb}$, thus yielding a well-defined metric on diagonal Hopf surfaces (cf. 
\S \ref{s:inducedsurfaces} for this terminology).  

These turns out to be the only compact non-K\"ahler fixed points of pluriclosed 
flow.  To show this classification, one first exploits a B\"ochner argument 
(\cite{GauduchonIvanov} Theorem 2) to show that for any Hermitian surface with 
$\rho_B^{1,1} = 0$, the Lee form is parallel.  If the Lee form vanishes, the 
metric is Calabi-Yau, and if not, the induced deRham splitting shows that the 
universal cover is isometric, up to scaling, to $\mathbb R \times S^3$ with the 
standard product metric.  One still has to identify the complex structure, and 
further work of Gauduchon (cf. \cite{GauduchonWeyl} III Lemma 11) shows that 
the 
only possible underlying complex surfaces are the standard Hopf surfaces, as 
claimed, and so the metrics as described above are the only examples.  See 
(\cite{GKRF}) for more detail.

\section{Classification of complex surfaces admitting solitons} 
\label{s:solitonclass}

In this section we give a classification of the possible complex surfaces 
admitting pluriclosed soliton structures.  To begin we define the correct 
notion of pluriclosed soliton, which is delicate due to the fact that it is 
related to the generalized Ricci flow via a nontrivial, and moreover 
non-gradient gauge transformation.  With this definition in place we show that 
a compact pluriclosed soliton on a surface is either K\"ahler-Einstein, or 
non-K\"ahler with the associated vector field $\theta^{\sharp} + \N f$ being a 
nontrivial holomorphic vector field.  With this in place we apply results from 
the Kodaira classification of surfaces to show that in the non-K\"ahler case 
the underlying complex surface must be a Hopf surface.

\subsection{Basic definitions} \label{ss:basicdefs}

The Ricci soliton equation is in part justified as the critical point equation 
for Perelman's $\gl$-functional, defined for arbitrary Riemannian metrics.  
Thus, in the K\"ahler setting, a priori one only expects the the Ricci soliton 
PDE (cf. Definition \ref{d:KRS}) to hold for the \emph{Riemannian metric}, i.e. 
not necessarily the K\"ahler form.  
However, by exploiting the K\"ahler condition, an elementary, well-known 
argument (cf. 
Proposition \ref{p:KRSholo}) shows that the gradient of the soliton function 
$f$ is automatically a holomorphic vector field, which also implies the 
K\"ahler form version of the soliton equation.  We include this simple argument 
for convenience as it indicates why an elementary 
adaptation to the pluriclosed flow setting is not possible.

\begin{defn} \label{d:KRS} A \emph{K\"ahler-Ricci soliton} is a K\"ahler 
manifold $(M^{2n}, 
g, J)$ together with $\gl \in \{-1,0,1\}$ and $f \in C^{\infty}(M)$ such that 
\begin{align*}
\Rc - \gl g = \N^2 f = L_{\tfrac{1}{2} \N f} g.
\end{align*}
\end{defn}

\begin{prop} \label{p:KRSholo} Let $(M^{2n}, g, J, \gl, f)$ be a K\"ahler-Ricci 
soliton.  Then
\begin{align} \label{KFsolitoneqn}
\rho - \gl \gw = L_{\tfrac{1}{2} \N f} \gw,
\end{align}
and $\tfrac{1}{2} \N f$ is the real part of a holomorphic vector field.
\begin{proof} As the Ricci tensor of a K\"ahler metric is $(1,1)$, it follows 
from the soliton equation that $(\N^2 f)^{2,0 + 0,2} = 0$.  Hence the Hessian 
is $(1,1)$, and we can moreover compute that
\begin{align*}
(\N^2 f)^{1,1}(\cdot, J \cdot)_{ij} =&\ \tfrac{1}{2} \left[ (\N^2 f)_{ik} + 
(\N^2 f)_{pq} J_i^p J_k^q \right] J_j^k \\
=&\ \tfrac{1}{2} \left[ J_j^k (\N^2 f)_{ik} - (\N^2 f)_{pj} J_i^p \right]\\
=&\ \tfrac{1}{2} \left[ \N_i (J_j^k \N_k f) - \N_j (J_i^k \N_k f) \right]\\
=&\ - \tfrac{1}{2} (d d^c f)_{ij}.
\end{align*}
On the other hand we can compute, using the Cartan formula and that $d \gw = 0$,
\begin{align*}
L_{\tfrac{1}{2} \N f} \gw =&\ d i_{\tfrac{1}{2} \N f} \gw + i_{\tfrac{1}{2} \N 
f} d \gw = d i_{\tfrac{1}{2} \N f} \gw.
\end{align*}
Finally we also have
\begin{align*}
(i_{\N f} \gw)_j =&\ \N^i f \gw_{ij} = g^{ip} d_p f g_{iq} J_j^q = - (d^c f)_j.
\end{align*}
Thus (\ref{KFsolitoneqn}) follows.  Now note the equations above imply
\begin{align*}
L_{\tfrac{1}{2}\N f} \gw = (L_{\tfrac{1}{2} \N f} g)(\cdot, J \cdot) = 
L_{\tfrac{1}{2} \N f} \gw - g(\cdot, L_{\tfrac{1}{2} \N f} J \cdot),
\end{align*}
and hence $L_{\tfrac{1}{2} \N f} J \equiv 0$, i.e. $\N f$ the real part of 
a holomorphic vector field.
\end{proof}
\end{prop}

In the pluriclosed case the story is different in a subtle way.  As explained 
in \S \ref{s:background}, it is necessary to apply a 
nontrivial 
gauge transformation to the pluriclosed flow to yield a solution of 
(\ref{f:GRF}), and it is \emph{that} system of equations which is the gradient 
flow of a modified Perelman functional (\cite{OSW}, cf. \cite{Streetsexpent}).  
Thus self-similar solutions to pluriclosed flow must satisfy the critical point 
equation for this modified Perelman functional.

\begin{defn} \label{d:PCS} Let $(M^{2n}, J)$ be a complex manifold.  We say 
that a pair $(g,f)$ of a Riemannian metric $g$ and $f \in C^{\infty}(M)$ is a 
\emph{pluriclosed steady soliton} if $g$ is a pluriclosed metric on $(M, J)$ and
\begin{gather} \label{f:soliton}
\begin{split}
\Rc - \tfrac{1}{4} H^2 =&\ \N^2 f = L_{\tfrac{1}{2} \N f} g,\\
\tfrac{1}{2} d^* H =&\ i_{\tfrac{1}{2} \N f} H.
\end{split}
\end{gather}
\end{defn}

Crucially, as the underlying metric is only pluriclosed, it does not follow 
from 
direct local calculations as in Proposition \ref{p:KRSholo} that the gradient 
of the soliton function $f$ is 
automatically holomorphic.  Nonetheless by applying classification results of 
Pontecorvo \cite{PontecorvoCS} on metrics compatible with several complex 
structures, we are able to show that a compact soliton on a complex surface, it 
is either K\"ahler-Einstein or the associated vector field $\theta^{\sharp} + 
\N f$ is nontrivial and holomorphic.

\begin{prop} \label{p:solitonholo}Let $(M^{4}, g, J)$ be a compact pluriclosed 
steady soliton.  Then 
either
\begin{enumerate}
\item $(M^4, g, J)$ is hyperHermitian, i.e. it is biholomorphically isometric 
to a flat torus, K3 surface with Calabi-Yau metric, or a hyperHermitian Hopf 
surface
\item The vector field $\theta^{\sharp} + \N f$ is a nontrivial holomorphic 
vector field on $M$.
\end{enumerate}

\begin{proof} Let $(\til{g}_t, \til{H}_t)$ denote the unique solution to 
(\ref{f:GRF}) with initial condition $(g, H)$.  Let $\psi_t$ denote the 
one-parameter family of diffeomorphisms generated by $- \N_{\til{g}_t} f$.  By 
a standard argument using the soliton equations (\ref{f:soliton}) we know that 
$(\til{g}_t, \til{H}_t) = (\psi^*_t g, \psi_t^* H)$.  On the other hand, let 
$g_t$ be the solution to pluriclosed flow with the given initial 
condition, with associated torsion $H_t = d^c \gw_t$.  Let $\phi_t$ be the 
one-parameter family of 
diffeomorphisms generated by $\theta_t^{\sharp}$.  It follows from 
(\cite{PCFReg}, Theorem 6.5), that $(\phi_t^* g_t, \phi_t^* H_t)$ is the unique 
solution to (\ref{f:GRF}) with initial condition $(g, H)$.  Thus
\begin{align*}
(\phi_t^* g_t, \phi_t^* H_t) = (\til{g}_t, \til{H}_t) = (\psi_t^* g, \psi_t^* 
H),
\end{align*}
and hence, setting $\Psi_t = \phi_t \circ \psi^{-1}_t$, we see that $(\Psi_t^* 
g_t, \Psi_t^* H_t) = (g,H)$.  Observe that since $g_t$ is a pluriclosed metric 
with respect to $J$, it is in particular compatible with $J$, and hence $g = 
\Psi_t^* g_t$ is compatible with $\Psi_t^* J$.  In particular, we have shown 
that the metric $g$ is compatible with a (possibly trivial) one-parameter 
family of complex 
structures.  If the family of complex structures is not stationary, 
then this yields a continuous family of complex structures compatible with $g$. 
 It follows from (\cite{PontecorvoCS} Theorem 5.5) that the metric is 
hyperHermitian, in particular it is either a flat torus, a K3 surface with 
Calabi-Yau metric or a hyperHermitian Hopf surface.  
 
 If this family is stationary, then by definition the generating 
vector field $\theta^{\sharp} + \N f$ is holomorphic.  If the vector field is 
trivial, then one has $\theta = - df$.  It follows from from the conformal 
transformation law for the Lee form (cf \cite{Gauduchon1form} \S I.13) that 
$\theta_{e^{f} g} = \theta_g + df = 0$.  In particular, the conformally related 
Hermitian 
metric $e^{f} g$ is K\"ahler.  But by the uniqueness of the Gauduchon metric 
in a fixed conformal class (\cite{Gauduchon1form} Main Theorem), it follows 
that, after possibly modifying $f$ by a 
constant, $g = e^{f} g$ and so $f \equiv 0$.  This means that the metric 
$g$ is K\"ahler, and Ricci flat, and we have reverted to the first case.
\end{proof}
\end{prop}

\subsection{Classification of compact steady solitons}

\begin{prop} \label{l:solKrigidity} Let $(M^{2n}, J)$ be a compact K\"ahler 
manifold, and suppose $(g,f)$ is a pluriclosed steady or shrinking soliton on 
$M$.  Then $(g,f)$ is 
a K\"ahler-Ricci soliton.
\begin{proof}  We first address the case of a steady soliton.  First we 
construct a particular 
$1$-form reduction of pluriclosed flow as in \cite{StreetsPCFBI} \S 3 and \S 4. 
 As the background manifold is K\"ahler, by a short argument (cf. 
\cite{ASNDGKCY} Proposition 6.1) using Hodge theory and the result of 
Demailley-Paun \cite{DemaillyPaun}, there exists $\ga \in 
\Lambda^{1,0}$ and a K\"ahler metric $\til{\gw}$ such that $\gw = \til{\gw} + 
\del \bga + \delb \ga$.  Next, fix an arbitrary Hermitian metric $h$, and for 
short time fix a background metric for the flow $\til{\gw}_t = \til{\gw} - t 
\rho_C(h)$.  Note that in the notation of \cite{StreetsPCFBI} we have $\mu = 
0$.  We then apply (\cite{StreetsPCFBI} Lemma 3.2) to construct a $1$-parameter 
family of $(1,0)$-forms $\ga_t$ which satisfy
\begin{align*}
\dt \ga = \delb^*_{g_t} \gw_t - \frac{\i}{2} \del \log \frac{\det g_t}{\det h}.
\end{align*}
A straightforward calculation shows then that $\gw_t = \til{\gw}_t + \del 
\bga_t + \delb \ga_t$ is the given solution to pluriclosed flow.  Furthermore, 
since $\del \til{\gw}_t = 0$ for all $t$, we can apply \cite{StreetsPCFBI} 
Proposition 4.10 with $\eta = 0$ to conclude that $\del 
\ga$ satisfies the evolution equation
\begin{align} \label{f:tpe}
\left(\dt - \gD_{g_t}^C \right) \brs{\del \ga}^2 \leq - \brs{T_{g_t}}^2,
\end{align}
where $T_{g_t}$ denotes the torsion of the Chern connection of the evolving 
metric.  But since the solution is a soliton, and hence evolves purely by 
diffeomorphism, there exists a vector field $X$ such that $\dt \brs{\del \ga}^2 
= X 
\brs{\del \ga}^2$.  Thus $\brs{\del \ga}^2$ is a subsolution of a strictly 
elliptic 
equation with zero constant term, and it follows from the strong maximum 
principle that it is constant. 
 It thus follows that $\brs{T_{g_t}}^2 = 0$, and so the metric is K\"ahler, and 
hence a K\"ahler-Ricci soliton, as claimed.

The case of a shrinking soliton is essentially the same.  We make choices of 
$\til{\gw}$ and $h$ as above, and this time set $\til{\gw}_t = \til{\gw} - t( 
\rho_C (h) - \til{\gw})$.  Following the arguments of (\cite{StreetsPCFBI} \S 
3) it is straightforward to construct a one-parameter family of $(1,0)$ forms 
which this time satisfy the $1$-form reduction of normalized pluriclosed flow, 
i.e.
\begin{align*}
\dt \ga = \delb^*_{g_t} \gw_t - \frac{\i}{2} \del \log \frac{\det g_t}{\det h} 
+ \ga.
\end{align*}
One also easily checks then that then the $1$-parameter family of metrics 
$\til{\gw}_t + \delb \ga_t + \del \bga_t$ is a solution of
\begin{align*}
\dt \gw = - \rho_B^{1,1} + \gw,
\end{align*}
the normalized pluriclosed flow.  Furthermore, an elementary adaptation of 
(\cite{StreetsPCFBI} Proposition 4.10) shows that $\del \ga$ will still satisfy 
(\ref{f:tpe}).  The reason no normalization terms are present is due to the 
fact that $\brs{\del \ga_t}^2_{g_t}$ has zero homogeneity in terms of the 
metric scaling.  As the solution to normalized pluriclosed flow follows a 
one-parameter family of diffeomorphisms, one can argue using (\ref{f:tpe}) 
exactly as above in the steady case to conclude that $\gw_t$ is K\"ahler, and 
hence a K\"ahler-Ricci soliton.
\end{proof}
\end{prop}

\begin{prop} \label{t:solclass} The following hold:
\begin{enumerate}
\item A shrinking pluriclosed soliton $(M^4, J, g, f)$ is a shrinking 
K\"ahler-Ricci soliton.
\item A steady pluriclosed soliton $(M^4, J, g, f)$ satisfies either
\begin{enumerate}
 \item $(M^4, J, g)$ is Calabi-Yau, and $f \equiv \mbox{const}$, or
 \item $(M^4, J)$ is biholomorphic to a Hopf surface.
\end{enumerate}
\end{enumerate}

\begin{proof} If the underlying surface $(M^4, J)$ is K\"ahler, Proposition 
(\ref{l:solKrigidity}) implies that $(g, f)$ is a K\"ahler-Ricci soliton (of 
course Calabi-Yau in the steady case), fitting into cases (1) and (2a) above.  
Thus we may assume that $(M^4, J)$ is non-K\"ahler.  We begin with the 
fundamental identity
\begin{align*}
 \tr_{\gw} \rho_B^{1,1} =&\ \tr_g \left( \Rc - \tfrac{1}{4} H^2 + 
L_{\theta^{\sharp}} g 
\right)\\
 =&\ \tr_g \left( L_{\theta^{\sharp} + \N f} g  + \gl g \right)\\
 =&\ \gD f + 4 \gl,
\end{align*}
since the Lee form is divergence free.
But also by (\cite{IvanovPapa} Proposition 3.3)
\begin{align*}
 \tr_{\gw} \rho_B^{1,1} = s_C - \brs{T}^2,
\end{align*}
where $s_C$ is the scalar curvature of the Chern connection.
Thus
\begin{align} \label{f:solclass10}
c_1 \cdot [\gw] =  \int_M s_C \gw^2 = \int_M \left(\brs{T}^2 + \gD f 
+ 4 \gl \right) dV_g = \int_M \brs{T}^2 dV_g + 4 \gl \Vol(g) > 0,
\end{align}
where the inequality is strict since $\gl \geq 0$ and the metric is not 
K\"ahler.  It follows from Gauduchon's plurigenera theorem 
(\cite{GauduchonFibres}, \cite{Gauduchon1form}) that $p_m = 0$ for all $m > 0$. 
 Hence $\Kod (M) = - \infty$, and $(M, J)$ is a Class VII surface.

Next we observe that,
since the pluriclosed flow with initial condition $g$ evolves by diffeomorphism 
pullback by a family of biholomorphisms and scaling by $\gl$, using Proposition 
\ref{p:solitonholo} one 
obtains, using that $s_C = \tr_{\gw} \rho_C$, where $\rho_C$ is the Chern-Ricci 
curvature,
\begin{align} \label{f:rigidity10}
 -\gl =&\ \dt \int_M s_C \gw^2 = \dt \int_M \rho_C \wedge \gw = - c_1^2.
\end{align}
Now note that it follows from (\cite{DlousskyNAC} Theorem 1.8) that for Class 
VII surfaces with $b_2 > 0$, one has $c_1^2 = - b_2$.  Since $\gl \geq 0$ this 
violates (\ref{f:rigidity10}), and hence we conclude $b_2 = 0$.  This implies 
moreover that $(M^4, J)$ is minimal, and then it follows that $c_1^2 \leq 0$ 
(cf. \cite{BHPV} \S VI).  Thus (\ref{f:rigidity10}) now forces both $\gl = 0$ 
and $c_1^2 = 0$.  Note that this has now ruled out the possibility of a 
shrinking soliton on a non-K\"ahler surface, finishing the proof of claim (1).  
To determine the biholomorphism type in the steady case, first note that by 
the classification of Class VII surfaces with $b_2 = 0$ 
(\cite{LiYauBog, TelemanBog}), $(M, J)$ is either a Hopf surface or an Inoue 
surface.  
It was shown 
in (\cite{Telemancone} Remark 4.2) that $c_1 \cdot [\gw] < 0$ for all metrics 
on 
Inoue surfaces, and so these cannot occur by (\ref{f:solclass10}).  Hence 
$(M^4, J)$ is biholomorphic to a minimal Hopf surface, as claimed.
\end{proof}
\end{prop}

\section{Sasakian and Hermitian geometry} \label{s:SasakiHermitian}

We now begin our construction of steady solitons on class $1$ Hopf surfaces.  
Thus construction builds upon a fundamental observation of 
Vaisman \cite{VaismanLCK} which exhibits a link between 
three-dimensional Sasakian structures and locally conformally K\"ahler metrics 
with parallel Lee form.  In this section we will briefly recall fundamental 
notions of Sasaki geometry, recall
the basic elements of Vaisman's construction, and discuss the relationship 
between the underlying Sasaki structure and the resulting complex surface.

\subsection{Basic definitions} 

\begin{defn} \label{d:sasaki} A Sasakian manifold consists of a triple $(M, g, 
Z)$ where $g$ is 
a Riemannian metric and $Z$ is a unit length Killing field with respect to $g$, 
such that $I \in \End(TM)$ defined via $I(X) := \N_X Z$ satisfies
\begin{align} \label{nablaPhi}
 (\N_X I)(Y) = g(Z,Y)X - g(X,Y)Z.
\end{align}
Associated to a Sasakian manifold we define $\eta = Z^{\flat}$, which is easily 
seen to satisfy.
\begin{align*}
\eta(Z) = 1, \qquad d \eta(Z, X) = 0.
\end{align*}
The kernel of $\eta$ is the \emph{transverse distribution}, which we will 
denote by $Q$, which comes equipped with a projection map
\begin{align} \label{f:pidef}
 \pi^T(X) := X - \eta(X) Z.
\end{align}
\end{defn}

\begin{prop} \label{p:sasakiprops}Let $(M, g, Z, \eta, I)$ be a
Sasakian manifold.  Then
\begin{enumerate}
 \item $I^2 Y = - Y + \eta(Y) Z$,
 \item $I Z = 0$, \qquad $\eta(I Y) = 0$,
 \item $g(X, I Y) - g(I X, Y) = 0$, \qquad $g(I Y, I X) = g(Y,X) -
\eta(Y) \eta(X),$
 \item $d \eta (Y,X) = 2 g(I Y,X)$,
 \item $L_{Z} I \equiv 0$.
\end{enumerate}
\begin{proof} We include some elementary derivations here for convenience and 
to fix conventions.
 
To prove (4) we compute using basic properties of the Levi-Civita connection 
and property (3),
\begin{align*}
 d \eta(X,Y) =&\ X \eta(Y) - Y \eta(X) - \eta[X,Y]\\
 =&\ X g(Z,Y) - Y g(Z,X) - g(Z,[X,Y])\\
 =&\ g( \N_X Z, Y) + g(Z, \N_X Y) - g(\N_Y Z, X) - g(Z,\N_Y X) - g(Z,[X,Y])\\
 =&\ g(IX, Y) - g(IY,X)\\
 =&\ 2 g(IX,Y).
\end{align*}

 To prove (5) we compute
 \begin{align*}
  (L_Z I)(X) =&\ L_Z (IX) - I(L_Z X)\\
  =&\ \N_Z (IX) - \N_{IX} Z - I([Z,X])\\
  =&\ (\N_Z I) X + I (\N_Z X) - I^2 X - I([Z,X])\\
  =&\ g(Z,X)Z - g(Z,X)Z + I(\N_X Z) + X\\
  =&\ 0,
 \end{align*}
as claimed.
\end{proof}
\end{prop}

\begin{defn} \label{d:transverse} Let $(M, g, Z)$ be a Sasakian manifold.  The
\emph{transverse metric} is defined by
\begin{align*}
 g^T(X,Y) = \tfrac{1}{2} d \eta(X, I Y).
\end{align*}
Similarly, define the \emph{transverse K\"ahler form} by 
\begin{align*}
 \gw^T(X,Y) = - \tfrac{1}{2} d \eta(X,Y).
\end{align*}
The reason for the factor $\tfrac{1}{2}$ in both formulas is explained by 
Proposition \ref{p:sasakiprops} (4), whereas the sign above is in keeping with 
the convention that the metric and K\"ahler form of a Hermitian structure 
satisfy $\gw(X,Y) = g(X,JY)$.  The terminology ``transverse'' arises from the 
fact that $g^T$ 
defines a positive definite metric on the distribution orthogonal to $Z$, 
whereas $g^T(Z, X) = 0$. 
\end{defn}

\subsection{The associated Hermitian cylinder}  
\label{s:hermcyl}

\begin{defn} \label{d:cylinderdef} Given $(N, g, Z)$ a Sasakian 
three-manifold, there 
is an associated Hermitian
manifold $(M, h, J)$ defined as follows.  Let
$M \cong N \times \mathbb R$, where we parameterize $\mathbb R$ with coordinate 
$t$, and set $W = \frac{\del}{\del t}$.  Choose the metric $h = \pi_1^* g + 
\pi_2^* dt^2$, and define $J$ via (recall $Q$ denotes the transverse 
distribution)
\begin{align*}
 J_{|Q} = I_{|Q}, \qquad J(Z) = W.
\end{align*}
We will refer to this Hermitian cylinder as a \emph{Sasaki-type complex 
surface}.  Moreover, we will refer to a tensor field $\mu$ on such a surface 
as \emph{invariant} if
\begin{align*}
L_{Z} \mu = L_W \mu = 0.
\end{align*}
\end{defn}

\begin{prop} \label{p:cylinderdef} Given a Sasakian three-manifold $(N,g,Z)$ 
the triple $(M, h, J)$
of Definition \ref{d:cylinderdef} satisfies:
\begin{enumerate}
\item The transverse projection map is holomorphic, i.e.
\begin{align*}
 \pi^T_* J X = I \pi_* X.
\end{align*}
\item $(M, h, J)$ is indeed Hermitian
\item The K\"ahler form associated to $(h, J)$ satisfies
\begin{align*}
 \gw_h = -\tfrac{1}{2} d \eta - dt \wedge \eta.
\end{align*}
\item The tensors $J$, $h$ and $\gw_h$, are all invariant.
\end{enumerate}
\begin{proof} Let $Q$ denote the transverse distribution.  As we can decompose 
$TM = Q + \IP{V} + \IP{W}$, given $X \in TM$ 
we may express $X = X_Q + \ga V + \gb W$, where $\ga,\gb \in \mathbb R$ and 
$X_Q \in Q$.  Then we compute
\begin{align*}
\pi_*^T JX = \pi_*^T (J (X_Q + \ga V + \gb W)) = \pi_*^T (I X_Q + \ga W - \gb 
V) = I X_Q = I \pi_*^T X,
\end{align*}
as required for property (1).

To check property (2), we first derive a formula for the metric $h$.  Combining 
properties (3) and (4) of Proposition \ref{p:sasakiprops} one obtains
\begin{align*}
 g = \tfrac{1}{2} d \eta \left(\cdot, I \cdot \right) + \eta \otimes \eta.
\end{align*}
Thus by definition, suppressing various obvious actions of $\pi_i^*$ we have
\begin{align} \label{f:hermcylinder10}
 h = \tfrac{1}{2} d \eta \left( \cdot, I \cdot \right) + \eta \otimes \eta + dt
\otimes dt.
\end{align}
So, it follows from the definition of $J$ that $dt \circ J = - \eta, \eta \circ 
J 
= dt$.  Also, by Proposition \ref{p:sasakiprops} items (3) and (4) it follows 
that $d \eta \in \Lambda^{1,1}(Q^*)$.  It follows easily that $h$ is Hermitian.

To show property (3), we pair equation (\ref{f:hermcylinder10}) with $J$ to 
obtain for arbitrary $X,Y \in TM$,
\begin{gather} \label{f:hermcylinder20}
\begin{split}
\gw_h(X ,Y) =&\ h(X, J Y)\\
=&\ \tfrac{1}{2} d \eta \left( X, I J Y \right) + \eta(X) \otimes \eta(J Y) + 
dt(
X) \otimes dt(J Y)\\
=&\ - \tfrac{1}{2} d \eta \left( X, Y \right) + \eta(X) \otimes dt (Y) - dt(X)
\otimes \eta(Y)\\
=&\ -\tfrac{1}{2} d \eta (X,Y) - dt \wedge \eta (X,Y),
\end{split}
\end{gather}
as claimed.  

Finally, to check the invariance by $Z$ we first recall from Definition 
\ref{d:sasaki} and Proposition \ref{p:sasakiprops} that 
$\eta, d \eta$, and $I$ are all $Z$ invariant on $N$, and thus also on $M$.  It 
is also apparent by 
construction that $dt$ is $Z$-invariant, and hence by (\ref{f:hermcylinder10}) 
and 
(\ref{f:hermcylinder20}) it follows immediately that $h$ and $\gw_h$ are 
$Z$-invariant.  Since $J = g^{-1} \gw_h$ it follows that $J$ is $Z$-invariant 
as well.  Invariance by $W$ is immediate.
\end{proof}
\end{prop}

\subsection{Induced surfaces} \label{s:inducedsurfaces}

The cylinder construction of  \S \ref{s:hermcyl} can yield many different 
complex surfaces, but we are here only interested in Hopf surfaces.  To recall, 
a Hopf surface is a compact complex surface whose universal covering space is
$\mathbb C^2 - \{(0,0)\}$.  A Hopf surface is \emph{primary} if $\pi_1(M) =
\mathbb Z$, and is otherwise \emph{secondary}.  For primary Hopf surfaces, 
Kodaira \cite{Kod2, Kod3} shows that the fundamental group is generated by a 
map $\gg$ defined by
\begin{align} \label{f:Hopfcontraction}
(z_1, z_2) \to (\ga z_1 , \gb z_2 + \gl z_1^m),
\end{align}
where $\ga, \gb, \gl$ are complex numbers satisfying $0 < \brs{\ga} \leq
\brs{\gb} < 1$ and
\begin{align*}
(\ga - \gb^m) \gl = 0.
\end{align*}
We say that a Hopf surface is of \emph{class 1} if $\gl = 0$, otherwise, it is 
\emph{class 0}.  Furthermore, we say the surface is \emph{diagonal} if $\gl = 
0$ and $\ga = \gb$.

We require explicit information on the construction of LCK metrics on Hopf 
surfaces, thus we briefly recall some 
elements of (\cite{Belgunmetric} \S 5), where this is carried out.  First we 
determine the holomorphic 
vector fields on Hopf surfaces.  To begin, express the generic holomorphic 
vector field on $\mathbb 
C^2 - \{(0,0)\}$ as $W = A(z_1,z_2) \del_{z_1} + B(z_1,z_2) \del_{z_2}$, where 
$A$ and $B$ are holomorphic functions.  By Hartogs' 
Theorem $A$ and $B$ extend to $\mathbb C^2$, and can be expressed as convergent 
power series.  To see which vector fields descend to the quotient, we must 
check invariance under the contraction (\ref{f:Hopfcontraction}).  For 
class 1 Hopf surfaces, i.e. $\gl = 0$, power series 
computations show that the general form of $A$ and $B$ for a 
$\gg$-invariant vector field is
\begin{gather} \label{f:vectfields}
\begin{split}
A(z_1,z_2) = a z_1 + b z_2, \ B(z_1,z_2) = c z_1 + d z_2, \quad a,b,c,d \in 
\mathbb C, \quad \ga = \gb\\
A(z_1,z_2) = a z_1 + c z_2^m, \ B(z_1,z_2) = b z_2, \quad a,b,c \in \mathbb 
C,\quad \ga^m = \gb\\
A(z_1,z_2) = a z_1, \ B(z_1,z_2) = b z_2, \quad a,b \in \mathbb C, \quad \ga 
\neq \gb\\
\end{split}
\end{gather}
As explained in (\cite{Belgunmetric} Proposition 7), the vector field $W$ 
playing 
the role of the lift of the parallel Lee vector field must satisfy the 
condition that $J W$ has relatively compact orbits in $\mathbb C^2 - 
\{(0,0)\}$.  By analyzing the orbits of $(1,0)$ and $(0,1)$ and comparing 
against (\ref{f:vectfields}), Belgun shows (\cite{Belgunmetric} Proposition 8) 
that the relevant vector fields are
\begin{align*}
W =&\ \Real \left\{ \ln \brs{\ga} z_1 \del_{z_1} + \ln \brs{\gb} z_2 \del_{z_2} 
\right\}\\
=&\ \tfrac{1}{2} \ln \brs{\ga} \left( x_1 \del_{x_1} + y_1 \del_{y_1} \right) + 
\tfrac{1}{2} \ln \brs{\gb}  \left( x_2 \del_{x_2} + y_2 \del_{y_2} \right).
\end{align*}
Of course then one has
\begin{align*}
Z = - J W = \tfrac{1}{2} \ln \brs{\ga} \left( y_1 \del_{x_1} - x_1 \del_{y_1} 
\right) + \tfrac{1}{2} \ln \brs{\gb} \left( y_2 \del_{x_2} - x_2 \del_{y_2} 
\right).
\end{align*}
As this $Z$ is the Reeb vector field of the associated Sasakian structure, we 
can derive the associated contact form.  First, for notational 
simplicity let
\begin{align*}
a = \tfrac{1}{2} \ln \brs{\ga}, \qquad b = \tfrac{1}{2} \ln \brs{\gb}, \qquad 
\gs = a \brs{z_1}^2 + b \brs{z_2}^2.
\end{align*}
It follows from elementary calculations that the associated contact form is
\begin{align*}
\eta_{\ga,\gb} =&\ \frac{1}{\gs} \left(x_1 dy_1 - y_1 dx_1 + x_2 dy_2 - y_2 
dx_2 \right).
\end{align*}
In our construction we also require a certain basis for the contact 
distribution.  To that end set
\begin{gather} \label{f:basisdef}
\begin{split}
E_1 =&\ \brs{z_2}^2 \left( y_1 \del_{x_1} - x_1 \del_{y_1} \right) - 
\brs{z_1}^2 \left( y_2 \del_{x_2} - x_2 \del_{y_2} \right)\\
E_2 =&\ J E_1 = \brs{z_2}^2 \left( x_1 \del_{x_1} + y_1 \del_{y_1} \right) - 
\brs{z_1}^2 \left( x_2 \del_{x_2} + y_2 \del_{y_2} \right).
\end{split}
\end{gather}
Again straightforward calculations yield
\begin{align*}
E_i \in \ker \eta_{\ga,\gb}, \qquad [E_i,Z] = 0.
\end{align*}

Furthermore, there is a natural dichotomy in this construction which will 
inform our construction.  If the orbits of the Reeb vector field of a Sasaki 
structure are all compact, then it generates a circle action.  If this action 
is 
free, the Sasakian structure is called \emph{regular}, and is called 
\emph{quasi-regular} if the action is locally free.  In case there is a 
noncompact orbit of the Reeb vector field, the Sasaki structure is called 
\emph{irregular}.  In our setting the only regular Sasaki structure has Reeb 
vector field tangent to the standard Hopf action on $S^3$.  The corresponding 
complex surface is the primary diagonal Hopf surface.  Furthermore in this 
setting the quotient space for the circle action is a smooth manifold, 
specifically $\mathbb {CP}^1$.  Quasi-regular Sasaki structures arise when one 
has $\ga^m = \gb^n$.  Here the corresponding complex surfaces are known as 
elliptic Hopf surfaces, and the quotient space has the structure of a bad 
orbifold, in particular is biholomorphic to one of the classic ``teardrop'' or 
``football'' orbifolds (see Figure \ref{f:orbpic}).
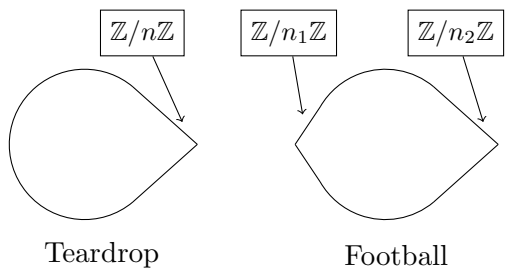
\begin{figure}[ht]
\begin{tikzpicture}

\draw (0,0) ++(-50:1) arc (-50:-310:1);
\draw (0,0) ++(-50:1) -> (1.5,0);
\draw (0,0) ++(-310:1) -> (1.5,0);

\draw (4,0) ++(50:1) arc (50:150:1);
\draw (4,0) ++(-50:1) arc (-50:-150:1);
\draw (4,0) ++(50:1) -> (5.5,0);
\draw (4,0) ++(-50:1) -> (5.5,0);
\draw (4,0) ++(-150:1) -> (2.8,0);
\draw (4,0) ++(150:1) -> (2.8,0);

\node at (0.2286,-1.4709) {Teardrop};
\node at (4.15,-1.4709) {Football};

\node [draw] (v3) at (0.7683,1.4825) {\small{$\mathbb Z/n \mathbb Z$}};
\draw [->] (v3) edge (1.3,0.3);

\node [draw] (v4) at (2.7156,1.4825) {\small{$\mathbb Z/n_1 \mathbb Z$}};
\draw [->] (v4) edge (2.9,0.4);

\node [draw] (v5) at (4.937,1.4825) {\small{$\mathbb Z/n_2 \mathbb Z$}};
\draw [->] (v5) edge (5.3,0.3);

\end{tikzpicture}
\caption{Bad orbifolds}
\label{f:orbpic}
\end{figure}
For elliptic Hopf surfaces we can solve for steady solitons on this base 
orbifold, and aspects of the construction have clear geometric meaning on this 
space.  In the case of Hopf surfaces corresponding to irregular Sasaki 
structures one is forced to work entirely on the total space, but analogies to 
the elliptic case persist in this setting.  As a final remark here we emphasize 
that Hopf surfaces of class $0$ do 
not appear via this construction, as exhibited in (\cite{Belgunmetric} 
Proposition 10).  It remains an open question whether class $0$ Hopf surfaces 
admit steady pluriclosed solitons or not, arising from some other construction.

\section{Invariant geometry on Hermitian cylinders} \label{s:invgeom}

While we have focused so far on the link between Sasakian and Hermitian 
geometry, we should not expect the pluriclosed flow equation to preserve any 
underlying connection to Sasakian geometry, due to the extra integrability 
conditions.  The curvature calculations to follow clarify this issue.  However, 
by standard arguments we obtain that invariance by vector fields is 
preserved by the flow (see Proposition \ref{p:invpreserved}).  In this section 
we thus investigate Sasaki-type Hermitian 
cylinders equipped with metrics which are invariant, but not necessarily coming 
from the Sasakian construction of \S \ref{s:SasakiHermitian}.  As we will see, 
the local geometry of invariant 
metrics is identical to that of a Kaluza-Klein type metric on a principal 
bundle.  Of course there is not necessarily a smooth quotient 
space for the action, and so we must work only on the total space.  

Given this one might wonder why we bother to pass through Sasaki geometry in 
the first place.  The reason is that one cannot apply standard averaging 
arguments to obtain an invariant metric in general for all of the complex 
surfaces under consideration, especially those corresponding to an irregular 
Sasaki structure, since the vector fields are not tangent to the action by a 
compact Lie group.  In other words, the construction of any invariant metric at 
all necessitates the construction via Sasakian geometry.  We work within the 
class of invariant metrics to produce a soliton.

\subsection{Characterization of invariant metrics} \label{ss:invchar}

\begin{defn} \label{Hcl} Given a Sasaki-type complex surface $(M^4, J)$, 
Consider the Lie algebra $\mathfrak t^2$ with basis $\{\xi_1,\xi_2\}$ and 
complex structure $J_{\mathfrak t^2} \xi_1 = \xi_2$.
We say that a form 
$\mu \in \Lambda^1 \otimes \mathfrak t^2$ is \emph{Hermitian connection} 
if, 
expressing
\begin{align*}
\mu(X) = \mu^Z(X) \xi_1 + \mu^W(X) \xi_2,
\end{align*}
one has
\begin{enumerate}
\item $\mu^Z(Z) \equiv \mu^W(W) \equiv 1, \qquad \mu^Z(W) \equiv \mu^W(Z) 
\equiv 0$,
\item $\mu(JX) = J_{\mathfrak t^2} \mu(X)$,
\item $L_{Z} \mu \equiv L_{T} \mu \equiv 0$.
\end{enumerate}
We will refer to the subbundle $\VV = \IP{V,W} \subset TM$ as the 
\emph{vertical 
space}.  Given a Hermitian connection form $\mu$ there is an associated 
complementary \emph{horizontal space} defined by $\HH = \ker \mu$.  Observe 
furthermore that condition (2), expanded out, says
\begin{align*}
\mu^Z(JX) \xi_1 + \mu^W(JX) \xi_2 = \mu(JX) = J_{\mathfrak t^2} \mu(X) = 
J_{\mathfrak t^2}( 
\mu^Z(X) \xi_1 + \mu^W(X) \xi_2) = - \mu^W(X) \xi_1 + \mu^Z(X) \xi_2,
\end{align*}
and so
\begin{align} \label{f:muJaction}
\mu^Z(JX) = - \mu^W(X), \qquad \mu^W(JX) = \mu^Z(X).
\end{align}
Also, $\mu$ defines natural projection operators
\begin{align*}
\pi_{\VV}(X) := \mu^Z(X) Z + \mu^W(X) W, \qquad \pi_{\HH} (X) := X - 
\pi_{\VV}(X).
\end{align*}
Lastly, we endow $\mathfrak t^2$ with the unique metric for which $\{\xi_1, 
\xi_2\}$ is an orthonormal basis, denoted $\IP{,}$.  In particular note that
\begin{align} \label{f:Liealgmetric}
\IP{\mu(X), \mu(Y)} = \mu^Z(X) \mu^Z(Y) + \mu^W(X) \mu^W(Y).
\end{align}
\end{defn}

\begin{lemma} \label{l:Hermitianbundle} Given a Sasaki-type complex surface 
$(M^4, J)$ and a Hermitian connection form, one has
\begin{align*}
[\VV,\VV] \subset \VV, \qquad [\VV,\HH] \subset \HH, \qquad  J \VV = \VV, \qquad
J \HH = \HH,
\end{align*}
\begin{proof} Since $\VV$ is spanned by $Z$ and $W$, which commute, the first 
inclusion follows
immediately, as does $J \VV = \VV$.  To show $[\VV,\HH] \subset \HH$, we 
observe that for $X \in \HH$,
\begin{align*}
\mu( [Z,X]) = \mu (L_Z X) = L_Z (\mu X) - (L_Z \mu) X = 0.
\end{align*}
To show that $J \HH = \HH$, since $J$ is invertible it suffices to
show that $J \HH \subset \HH$.  Suppose that there existed $X \in \HH$ such that
$\pi_{\VV} J X \neq 0$.  Then since $J$ is an invertible endomorphism on $\VV$, 
it
follows that $0 \neq \pi_{\VV} J J X = - \pi_{\VV} X$, a contradiction. 
\end{proof}
\end{lemma}

\begin{prop} \label{p:KKmetric} Given a Sasaki-type complex surface $(M^4, J)$, 
an invariant 
Hermitian metric $g$ is equivalent to an invariant triple $(g^T, \mu, \psi)$ 
where 
$g^T$ is a transverse Hermitian metric, $\mu$ is a Hermitian connection 
form, and $\psi \in C^{\infty}(M)$ is a positive function.
\begin{proof} This is essentially the standard argument decomposing an 
invariant metric on a principal bundle, just without the existence of a smooth 
quotient space.  First, given a triple $(g^T, \mu, \psi)$ as in the statement, 
we 
recover $g$ as
\begin{align*}
g(X,Y) = g^T (X, Y) + \psi \IP{\mu(X), \mu(Y)}.
\end{align*}
To obtain the decomposition given an invariant metric $g$, first observe using 
the Hermitian property that we may define
\begin{align} \label{f:fdef}
\psi := g(Z,Z) = g(W,W),
\end{align}
which is positive since $g$ is positive definite.  We check that $\psi$ is 
invariant using the Leibniz rule for Lie derivatives to 
obtain
\begin{align*}
L_{Z} \psi = L_{Z} g(Z,Z) = (L_Z g)(Z,Z) + 2 g([Z,Z],Z) = 0.
\end{align*}
Similarly $L_{W} \psi = 0$.  

Next we define a Hermitian connection form via
\begin{align} \label{f:mudef}
\mu^{Z}(X) := \psi^{-1} g(Z,X), \qquad \mu^W(X) := \psi^{-1} g(W,X).
\end{align}
We check the algebraic conditions for $\mu$ to define a Hermitian 
connection form.  First, it follows directly from (\ref{f:fdef}) and 
(\ref{f:mudef}) that $\mu^Z(Z) \equiv \mu^W(W) \equiv 1$.  Since $g$ is 
Hermitian we conclude
\begin{align*}
\mu^Z(W) \equiv \psi^{-1} g(Z,W) \equiv \psi^{-1} g(Z,JZ) \equiv 0,
\end{align*}
and similarly $\mu^W(Z) \equiv 0$, finishing condition (1).  To check condition 
(2), we choose any vector $X \in TM$ and then compute
\begin{align*}
\mu(JX) =&\ \mu^Z(JX) \xi_1 + \mu^W(JX) \xi_2\\
=&\ \psi^{-1} g(Z,JX) \xi_1 + \psi^{-1} g(W,JX) \xi_2\\
=&\ - \psi^{-1} g(W,X) \xi_1 + \psi^{-1} g(Z,X) \xi_2\\
=&\ \mu^W(X) \xi_1 - \mu^Z(X) \xi_2\\
=&\ J_{\mathfrak t^2} \mu(X).
\end{align*}
To check the invariance, we build upon the invariance of $\psi$ to obtain
\begin{align*}
(L_{Z} \mu^Z)(X) =&\ \psi^{-1} \left( L_Z (\mu^Z(X)) - \mu^Z(L_Z X) \right)\\
=&\ \psi^{-1} (L_Z g(Z,X) - g(Z, [Z,X]))\\
=&\ 0.
\end{align*}
The invariance by $W$ is proved similarly.

Lastly we define
\begin{align*}
g^T(X,Y) = g(X,Y) - \psi \IP{\mu(X), \mu(Y)}.
\end{align*}
The invariance follows directly from the previous invariance claims.  Next we 
check that it is indeed transverse, i.e.
\begin{align*}
g^T(X,Z) =&\ g(X,Z) - \psi \IP{ \mu(X), \mu(Z)}\\
=&\ \psi \mu^Z(X) - \psi \mu^Z(X)\\
=&\ 0.
\end{align*}
We also check that $g^T$ is a $(1,1)$-tensor.  To that end,
\begin{align*}
g^T(JX,JY) =&\ g(JX,JY) - \psi \IP{\mu(JX),\mu(JY)}\\
=&\ g(X,Y) - \psi \IP{J_{\mathfrak t^2} \mu(X), J_{\mathfrak t^2} \mu(Y)}\\
=&\ g^T(X,Y).
\end{align*}
\end{proof}
\end{prop}

\begin{defn} \label{d:transKF} Given a Sasaki-type complex surface $(M^4, J)$ 
and an invariant 
Hermitian metric $g$, define the \emph{transverse K\"ahler form} via
\begin{align} \label{f:transKF}
\gw^T(X,Y) =&\ g^T(X,JY).
\end{align}
Similarly define the \emph{vertical metric} via $g^V(X,Y) = g(X,Y) - g^T(X,Y) = 
\psi 
\IP{\mu(X),\mu(Y)}$, with associated vertical K\"ahler form
\begin{gather} \label{f:vertKF}
\begin{split}
\gw^V(X,Y) =&\ g(X,JY) - g^T(X,JY)\\
=&\ \psi \IP{\mu(X), \mu(JY)}\\
=&\ \psi ( - \mu^Z(X) \mu^W(Y) + \mu^W(X) \mu^Z(Y))\\
=&\ \psi \mu^W \wedge \mu^Z(X,Y).
\end{split}
\end{gather}
\end{defn}

\subsection{Torsion of invariant metrics}

Here we investigate the structure of the torsion of invariant Hermitian 
metrics.  The principal observation, deduced in Lemma \ref{metriclemma20}, is 
that 
such a metric is pluriclosed if and only if the fiber length function $\psi$ is 
constant.  We begin with the definition of the curvature of a Hermitian 
connection form.

\begin{defn} \label{d:horiz} Given a Sasaki-type complex surface $(M^4, J)$ and 
a Hermitian connection form $\mu$, the \emph{curvature} $F \in 
\Lambda^2(M) \otimes \mathfrak t^2$ is
\begin{align*}
F = d \mu = F^Z \xi_1 + F^W \xi_2.
\end{align*}
Notice that, as sections of $\Lambda^2(\HH^*)$, both $F^Z$ and $F^W$ are type 
$(1,1)$ since $\HH$ is of real rank $2$.
\end{defn}

\begin{lemma} \label{metriclemma20}  Given a 
Sasaki-type complex surface $(M^4, 
J)$, an invariant 
Hermitian metric $g$ satisfies
\begin{enumerate}
\item $d \gw^T = 0$,
\item $d \gw^V = d\psi \wedge \mu^W \wedge \mu^Z + \psi d \mu^W \wedge \mu^Z - 
\psi 
\mu^W \wedge d \mu^Z$,
\item  $d^c \gw_g = - d^c \psi \wedge \mu^Z \wedge \mu^W + \psi d \mu^W \wedge 
\mu^W 
+ \psi \mu^Z 
\wedge d \mu^Z$,
\item $d d^c \gw_g = \frac{d d^c \psi}{\psi} \wedge \gw_g.$
\end{enumerate}
In particular, $g$ is pluriclosed if and only if $\psi \equiv \mbox{const}$.
\begin{proof} For item (1), first note that since both $\HH$ and $\VV$ are rank 
$2$, evaluating any three form purely on vectors of one type or the other 
yields zero.  Now choose $X,Y \in \HH$ and $Z \in \VV$ we obtain
\begin{align*}
d \gw^T(X,Y,Z) =&\ X \gw^T(Y,Z) + Y \gw^T(Z,X) + Z \gw^T(X,Y)\\
&\ - \gw^T ([X,Y],Z) 
+ \gw^T([X,Z],Y) - \gw^T([Y,Z],X)\\
=&\ Z \gw^T(X,Y) + \gw^T([X,Z],Y) - \gw^T([Y,Z],X)\\
=&\ L_Z(\gw^T(X,Y)) - \gw^T( L_Z X, Y) - \gw^T (X,L_Z Y)\\
=&\ (L_Z \gw^T)(X,Y) + \gw^T(X, L_Z Y) - \gw^T(X, L_Z Y)\\
=&\ 0,
\end{align*}
as required.  Similarly, choosing $X \in \HH$ and $Z,W \in \VV$ we compute, 
using that $\VV$ is integrable,
\begin{align*}
d \gw^T(X,Z,W) =&\ X \gw^T(Z,W) + W \gw^T(X,Z) + Z \gw^T(W,X)\\
&\ - \gw^T([X,Z],W) + \gw^T([X,W],Z) - \gw^T([Z,W],X)\\
=&\ 0,
\end{align*}
and so $d \gw^T = 0$.

We compute $d \gw^{V}$ using (\ref{f:vertKF}) to immediately yield
\begin{align*}
d \gw^V =&\ d (\psi \mu^W \wedge \mu^Z) = d \psi \wedge \mu^W \wedge \mu^Z + 
\psi d \mu^W 
\wedge \mu^Z - \psi \mu^W \wedge d \mu^Z,
\end{align*}
as claimed.

Next, we compute, using (1), (2), (\ref{f:muJaction}), and the fact that $F$ is 
of type $(1,1)$,
\begin{align*}
 d^c \gw_g =&\ - d \gw_g(J,J,J)\\
 =&\ - d \gw^V(J, J, J)\\
 =&\ - \left\{ (d\psi \circ J) \wedge \mu^Z \wedge (-\mu^W) + \psi d \mu^W 
\wedge (- 
\mu^W) - \psi (\mu^Z) \wedge d \mu^Z \right\}\\
 =&\ - d^c \psi \wedge \mu^Z \wedge \mu^W + \psi d \mu^W \wedge \mu^W + \psi 
\mu^Z 
\wedge d \mu^Z,
\end{align*}
as claimed.

Differentiating again we obtain, using that $d \mu^Z, d \mu^W$ and $\psi$ are 
invariant and 
comparing against (\ref{f:vertKF}),
\begin{align*}
d d^c \gw_g =&\ d \left[ - d^c \psi \wedge \mu^Z \wedge \mu^W + \psi d \mu^W 
\wedge 
\mu^W - \psi \mu^Z \wedge d \mu^Z \right]\\
 =&\ - d d^c \psi \wedge \mu^Z \wedge \mu^W\\
 =&\ \frac{ d d^c \psi}{\psi} \wedge \gw^V\\
 =&\ \frac{ d d^c \psi}{\psi} \wedge \gw_g,
 \end{align*}
 where the last line follows because $d d^c \psi$ is horizontal, and hence $d 
d^c 
\psi \wedge \gw^T = 0$.  This finishes (4).
 
Using this formula and multiplying by $\psi > 0$, we see that $d d^c \gw_g = 0$ 
if 
and only if $0 = d d^c \psi \wedge \gw_g = (\tr_{\gw} \i \del \delb \psi) 
\gw_g^2$, 
if and only if $\gD_C \psi = 0$.  It thus follows from a standard maximum 
principle 
argument that $d d^c \gw_g = 0$ if and only if $\psi$ is constant, as claimed.
\end{proof}
\end{lemma}

\subsection{Bismut curvature}

In this subsection we establish formulas for the Bismut curvature of a 
pluriclosed 
invariant metric on a Sasaki-type complex surface.  We begin with a basic lemma 
producing a frame canonically associated to any point.

\begin{lemma} \label{adaptedframe}  Given a pluriclosed 
invariant Hermitian metric $g$
on a Sasaki-type complex surface $(M^4, J)$, for each $p \in M$ there exist 
local coordinate
vector fields $\{\del_{x_1}, \del_{x_2}\}$ such that
\begin{align*}
\mu(\del_{x_i}) (p) = 0, \qquad J \del_{x_1} = \del_{x_2}, \qquad [\del_{x_i}, 
Z] = [\del_{x_i},W] = 0.
\end{align*}
Furthermore, the vector fields defined by
\begin{align*}
 e_i = \del_{x_i} - \mu^Z(\del_{x_i}) Z - \mu^W(\del_{x_i}) W,
\end{align*}
satisfy
\begin{align*}
\mbox{span} \{e_1, e_2\} = \HH, \qquad J
e_1 = e_2.
\end{align*}
Furthermore one has
\begin{align*}
 [e_1, e_2] =&\ - F^Z (\del_{x_1}, \del_{x_2}) Z - F^W (\del_{x_1}, \del_{x_2}) 
W, \quad [e_i,
Z] = [e_i, W] = 0.
\end{align*}
\begin{proof} Given $p \in M$, we fix a Hermitian basis for $\HH_p$, then 
extend 
this using $\{Z_p,W_p\}$ to yield a Hermitian basis for $TM_p$.  By standard 
arguments this basis can be extended locally to a complex coordinate frame for 
$TM$, with the extended basis including $\{V,W\}$ locally.  Choosing the vector 
fields associated to the initially spanning vectors for $\HH_p$ yields 
$\{\del_{x_1}, \del_{x_2}\}$ which by construction satisfy the three claimed 
conditions.

It is clear by construction that $e_1, e_2$ are horizontal and linearly
independent, hence span $\HH$.  Moreover, we note using (\ref{f:muJaction})
that
\begin{align*}
J e_1 =&\ J \left( \del_{x_1} - \mu^Z(\del_{x_1}) Z -
\mu^W(\del_{x_2}) W \right)\\
=&\ \del_{x_2} - \mu^Z(\del_{x_1}) W + \mu^W(\del_{x_1})
Z\\
=&\ \del_{x_2} + \mu^Z(J \del_{x_2}) W - \mu^W(J \del_{x_2})
Z\\
=&\ \del_{x_2} - \mu^Z(\del_{x_2}) Z  - \mu^W (\del_{x_2})
W\\
=&\ e_2.
\end{align*}
Now we compute the commutators.  Note using invariance of the connection that
\begin{align*}
 Z \mu(\del_{x_i}) =&\ (L_{Z} \mu)(\del_{x_i}) + \mu([Z,\del_{x_i}]) =
0,
\end{align*}
and similarly for $W$.  Thus
\begin{align*}
[e_1,e_2] =&\ [\del_{x_1} - \mu^Z(\del_{x_1}) Z - \mu^W(\del_{x_1}) W, 
\del_{x_2} -
\mu^Z(\del_{x_2}) Z - \mu^W(\del_{x_2}) W]\\
=&\ \left\{ \del_{x_2} \mu^Z(\del_{x_1}) - \del_{x_1} \mu^Z(\del_{x_2}) 
\right\} 
Z + \left\{ \del_{x_2} \mu^W(\del_{x_1}) - \del_{x_1} \mu^W(\del_{x_2}) 
\right\} 
W\\
=&\ - F^Z (\del_{x_1}, \del_{x_2}) Z - F^W (\del_{x_1}, \del_{x_2}) W.
\end{align*}
The remaining vanishing claims are immediate.
\end{proof}
\end{lemma}

\begin{rmk} \label{r:frameremarks} Without further explicitly invoking Lemma 
\ref{adaptedframe}, given a pluriclosed invariant Hermitian metric $g$ we will henceforth 
ask for ``an adapted frame'' at a point $p \in (M^4, J)$ a Sasaki-type complex 
surface.  This will mean the vector fields $\{e_i\}$ constructed in Lemma 
\ref{adaptedframe}, augmented with $\{Z,W\}$ to yield a local frame.  In 
computations we will use lowercase Roman letters to refer to the vectors 
$\{e_i\}$, and Greek letters $\{e_{\ga}\}, \ga = 1,2$, to refer to the vectors 
$\{V,W\}$.  Furthermore we use uppercase Roman letters to indicate a general 
element of the overall frame.

Lastly, we note that in the the construction of Lemma \ref{adaptedframe}, by a 
linear change of coordinates on $\{e_1,e_2\}$ in the transverse direction, we 
can assume without loss of generality that $e_i g(e_j,e_k)(p) = 0$.  Using this 
it is easy to see that one then has $e_A g(e_B,e_C)(p) = 0$, so that all first 
derivatives of the metric with respect to the frame vanish at $p$.
\end{rmk}

\begin{lemma} \label{Christoffel}  Given a pluriclosed 
invariant Hermitian metric $g$
on a Sasaki-type complex surface $(M^4, J)$, $p \in M$, and an adapted frame at 
$p$, one has
\begin{align*}
 ^{g}\gG_{ijk} =&\ ^{g^T}\gG_{ijk},\\
 ^{g}\gG_{ij\ga} =&\ - \tfrac{1}{2} F_{ij\ga},\\
  ^g\gG_{i \ga j} =&\ ^g \gG_{\ga i j} = \tfrac{1}{2} F_{ij \ga},
\end{align*}
and all other Christoffel symbols vanish.
\begin{proof} Recall the basic formula
\begin{align*}
\IP{\N_X Y,Z} =&\ \tfrac{1}{2} \left\{ X \IP{Y,Z} + Y\IP{X,Z} - Z\IP{X,Y} -
\IP{[X,Z],Y} - \IP{[Y,Z],X} - \IP{[Y,X],Z} \right\}.
\end{align*}
Using Lemma \ref{adaptedframe}, since the Lie bracket of any two horizontal
fields in the frame is vertical, the equality $^{g}\gG_{ijk} =\ ^{g^T}\gG_{ijk}$
follows immediately.  To compute $^g\gG_{ij\ga}$, using the invariance 
properties of the metric and the frame the
only possible remaining terms are the Lie bracket terms, which using Lemma
\ref{adaptedframe} yields
\begin{align*}
 ^g\gG_{ij\ga} =&\ \tfrac{1}{2} \IP{[e_i,e_j], e_{\ga}} = - \tfrac{1}{2}
g_{\ga \gb} F^{\gb}_{ij} = - \tfrac{1}{2} F_{ij \ga}.
\end{align*}
Similarly
\begin{align*}
 ^g\gG_{i \ga j} =&\ - \tfrac{1}{2} \IP{ [e_i,e_j], e_{\ga}} = 
\tfrac{1}{2} g_{\ga \gb} F^{\gb}_{ij} = \tfrac{1}{2} F_{ij\ga}.
\end{align*}
\end{proof}
\end{lemma}

\begin{lemma} \label{torsionlemma}  Given a pluriclosed 
invariant Hermitian metric $g$
on a Sasaki-type complex surface $(M^4, J)$, $p \in M$ and an adapted frame at 
$p$, one has
\begin{align*}
 (d^c \gw)_{ijk} = (d^c \gw)_{i \ga \gb} = 0, \qquad (d^c \gw)_{ij\ga} = 
F_{ij\ga}.
\end{align*}

\begin{proof} We recall $H = \tfrac{1}{2} d^c \gw$.  We first make a general
calculation.  Since $\gw(X,Y) = g(X,JY)$, we see
\begin{align*}
d^c \gw(X,Y,Z) =&\ - d \gw(JX, JY, JZ)\\
=&\ - \left\{ JX \gw(JY, JZ) + JZ \gw (JX,JY) + JY \gw( JZ, JX) \right.\\
&\ \left.  - \gw ([JX,JY], JZ) + \gw ([JX,JZ], JY) - \gw([JY,JZ],JX) \right\}\\
=&\ JX g( JY, Z) + JZ g( JX, Y) + JY g(JZ, X)\\
&\ - g( [JX, JY], Z) + g ([JX,JZ],Y) - g([JY,JZ],X).
\end{align*}
Using this and Lemma \ref{adaptedframe} it is clear that $H_{ijk} = H_{i \ga 
\gb} = 0$, and moreover
\begin{align*}
 (d^c \gw)_{12 \ga} =&\ - g([J e_1, J e_2],e_{\ga} ) = g ([ e_2, e_1],
e_{\ga}) = F_{12 \ga}.
\end{align*}
\end{proof}
\end{lemma}

\begin{lemma} \label{Bismutchris}  Given a pluriclosed 
invariant Hermitian metric $g$
on a Sasaki-type complex surface $(M^4, J)$, $p \in M$ and an adapted frame at 
$p$, one has
\begin{align*}
^B \gG_{ijk} =&\ ^{g}\gG_{ijk},\\
^B \gG_{\ga i j} =&\ F_{ij\ga},
\end{align*}
and all other Christoffel symbols vanish.
\begin{proof} These formulas follow directly from the above Lemmas 
\ref{Christoffel} and \ref{torsionlemma} and the 
formula
\begin{align*}
^B \gG_{ABC} =\ ^g \gG_{ABC} + \tfrac{1}{2} (d^c \gw)_{ABC}.
\end{align*}
\end{proof}
\end{lemma}

\begin{prop} \label{STrhoB}  Given a pluriclosed invariant 
Hermitian metric $g$
on a Sasaki-type complex surface $(M^4, J)$, $p \in M$, and an adapted frame at 
$p$, one has
\begin{align*}
\rho_B(e_1,e_2) =&\ \tfrac{1}{2} \left(- R^T + \brs{F}^2 \right)\\
\rho_B(e_1,Z) =&\ \left( d \tr_{\gw} F^Z \right)_1\\
\rho_B(e_1,W) =&\ \left( d \tr_{\gw} F^W \right)_1\\
\rho_B(e_2,Z) =&\ \left( d \tr_{\gw} F^Z \right)_2\\
\rho_B(e_2,W) =&\ \left( d \tr_{\gw} F^W \right)_2\\
\rho_B(Z,W) =&\ 0.
\end{align*}
 \begin{proof} We will drop the notation ``$B$'' from the connection 
coefficients and curvature tensor throughout this proof, and so $\gG$ and 
$\Omega$ are associated to the Bismut connection.  To begin with recall the 
basic formula
 \begin{align} \label{Bismutcurv}
 \Omega_{ABC}^D =&\ e_{A} \gG_{BC}^D - e_B \gG_{AC}^D + \gG_{BC}^E \gG_{AE}^D - 
\gG_{AC}^E \gG_{BE}^D - [e_A,e_B]^E \gG_{EC}^D.
 \end{align}
 Since all first derivatives of the metric coefficients with respect to the 
frame vanish at $p$ (cf. Remark \ref{r:frameremarks}), we obtain the formula
 \begin{align*}
 \Omega(A,B,C,D) = e_A \gG_{BCD} - e_B \gG_{ACD} + \gG_{BC}^E \gG_{AED} - 
\gG_{AC}^E \gG_{BED} - [e_A,e_B]^E \gG_{ECD}.
 \end{align*}
Next we observe the general calculation
 \begin{align*}
 \rho_B(X,Y) =&\ \tfrac{1}{2} \sum_{i=1}^4 \Omega(X,Y, e_i, J e_i)\\
 =&\ \Omega(X,Y, e_1, e_2) + \Omega(X,Y, Z, W)\\
=&\ \Omega(X,Y, e_1, e_2), 
  \end{align*}
  where the last line follows because every Christoffel symbol of the form 
$\gG_{AB\ga}$ vanishes by Lemma \ref{Bismutchris}.  First we can compute
  \begin{align*}
  \rho_B(e_1,e_2) =&\ \Omega(e_1,e_2,e_1,e_2)\\
  =&\ e_1 \gG_{212} - e_2 \gG_{112} + \gG_{2 1}^E \gG_{1 E 2} - \gG_{11}^E 
\gG_{2 E 2} - [e_1,e_2]^E \gG_{E12}\\
  =&\ \Rm^{g^T}_{1212} + F_{12\ga} F_{12\ga}\\
  =&\ \tfrac{1}{2} \left(- R^T + \brs{F}^2 \right).
  \end{align*}
Next we see
\begin{align*}
\rho_B(e_1,Z) =&\ e_1 \gG_{Z12} - Z \gG_{112} + \gG_{Z1}^E \gG_{1 E 2} - 
\gG_{1 1}^E \gG_{Z E 2}= e_1 (F_{12Z}) = \left( d \tr_{\gw} F^Z \right)_1.
\end{align*}
Similarly
\begin{align*}
\rho_B(e_1,W) =&\ e_1 \gG_{W12} - W \gG_{112} + \gG_{W1}^E \gG_{1E 2} - 
\gG_{11}^E \gG_{WE2} = e_1 (F_{12W}) = \left( d \tr_{\gw} F^W \right)_1.
\end{align*}
Similarly
\begin{align*}
\rho_B(e_2,Z) =&\ e_2 \gG_{Z12} - Z \gG_{112} + \gG_{Z1}^E \gG_{2E2} - 
\gG_{21}^E \gG_{ZE2} = e_2 (F_{12Z}) = \left( d \tr_{\gw} F^Z \right)_2.
\end{align*}
Similarly
\begin{align*}
\rho_B(e_2,W) = e_2 \gG_{W12} - W \gG_{112} + \gG_{W1}^E \gG_{2E2} - 
\gG_{21}^E \gG_{WE2} = e_2 (F_{12W}) = \left( d \tr_{\gw} F^W \right)_2.
\end{align*}
Lastly we see, using the invariance properties and Lemma \ref{Bismutchris},
\begin{align*}
\rho_B(Z,W) = Z \gG_{W12} - W \gG_{Z12} + \gG_{W1}^E \gG_{ZE2} - 
\gG_{Z1}^E \gG_{WE2} = F_{1 E W} F_{E 2 Z} - F_{1EZ} F_{E2W} = 0.
\end{align*}
\end{proof}
\end{prop}

\subsection{Lie derivative operators}

\begin{lemma} \label{l:Leeformlemma}  Given a pluriclosed 
invariant Hermitian metric $g$
on a Sasaki-type complex surface $(M^4, J)$, $p \in M$, and an adapted frame at 
$p$, one has
\begin{align*}
\theta(e_i) =&\ 0,\\
\theta(Z) =&\ - \tr_{\gw} F^W,\\
\theta(W) =&\ \tr_{\gw} F^Z.
\end{align*}
\begin{proof} With respect to an arbitrary frame the Lee form can be expressed 
as
\begin{align*}
\theta_A =&\ - \tfrac{1}{2} (d^c \gw)_{BCD} J_A^B g^{CE} J_E^D .
\end{align*}
Comparing this against the result of Lemma \ref{torsionlemma}, it is clear that 
$\theta(e_i) = 0$.  Moreover,
\begin{align*}
\theta(Z) = - \tfrac{1}{2} (d^c \gw)_{W C D} g^{CE} J_E^D = - \tfrac{1}{2} 
\left\{ (d^c \gw)_{W 1 2} - (d^c \gw)_{W21} \right\} = - \tr_{\gw} F^W.
\end{align*}
Similarly
\begin{align*}
\theta(W) = \tfrac{1}{2} (d^c \gw)_{Z C D} g^{CE} J_E^D = \tfrac{1}{2} \left\{ 
(d^c \gw)_{Z 1 2} - (d^c \gw)_{Z21} \right\} = \tr_{\gw} F^Z,
\end{align*}
as required.
\end{proof}
\end{lemma}

\begin{lemma} \label{l:LeeformLDlemma}  Given a pluriclosed 
invariant Hermitian metric $g$
on a Sasaki-type complex surface $(M^4, J)$, $p \in M$, and an adapted frame at 
$p$, one has
\begin{align*}
\left(L_{\theta^{\sharp}} g\right) (e_i,e_j) =&\ 0,\\
\left(L_{\theta^{\sharp}} g\right) (e_i,Z) =&\ -  e_i (\tr_{\gw} F^W),\\
\left(L_{\theta^{\sharp}} g\right) (e_i,W) =&\ e_i (\tr_{\gw} F^Z),\\
\left(L_{\theta^{\sharp}} g\right) (e_{\ga},e_{\gb}) =&\ 0.
\end{align*}

\begin{proof} We first observe the general formula
\begin{align*}
(L_X g)(Y,Z) = X g(Y,Z) - g(L_X Y, Z) - g(Y, L_X Z).
\end{align*}
Using this and the properties of our adapted frame we see
\begin{align*}
(L_{\theta^{\sharp}} g) (e_i,e_j) =&\ \theta^{\sharp} g(e_i,e_j) - g([ 
(\tr_{\gw} F^Z) W - (\tr_{\gw} F^W) Z, e_i],e_j)\\
&\ - g(e_i, [ (\tr_{\gw} F^Z) W - (\tr_{\gw} F^W) Z, e_j])\\
=&\ 0.
\end{align*}
Next we have
\begin{align*}
(L_{\theta^{\sharp}} g) (e_i,Z) =&\ \theta^{\sharp} g(e_i,Z) - g([\tr_{\gw} 
F^Z) W - (\tr_{\gw} F^W) Z, e_i],Z)\\
&\ - g(e_i, [(\tr_{\gw} F^Z) W - (\tr_{\gw} F^W) Z, Z])\\
=&\ - e_i (\tr_{\gw} F^W).
\end{align*}
Similarly
\begin{align*}
(L_{\theta^{\sharp}} g) (e_i,W) =&\ \theta^{\sharp} g(e_i,W) - g([\tr_{\gw} 
F^Z) W - (\tr_{\gw} F^W) Z, e_i],W)\\
&\ - g(e_i, [ (\tr_{\gw} F^Z) W - (\tr_{\gw} F^W) Z, W])\\
=&\ e_i (\tr_{\gw} F^Z).
\end{align*}
Lastly, using the invariance properties and the general formula for $L_X 
g(Y,Z)$ above it is clear that $L_{\theta^{\sharp}} g(Z,Z) = 
L_{\theta^{\sharp}} g(W,W) = L_{\theta^{\sharp}} g(Z,W) = 0$.
\end{proof}
\end{lemma}

\begin{lemma} \label{l:Hessiandecomp}  Given a pluriclosed
invariant Hermitian metric $g$
on a Sasaki-type complex surface $(M^4, J)$, $p \in M$, and an adapted frame at 
$p$, one has for $f$ an invariant function,
\begin{align*}
\N^2 f(e_i,e_j) =&\ (\N^T)^2 f (\del_{x_i}, \del_{x_j}),\\
\N^2 f(e_i,e_{\ga})=&\ - \tfrac{1}{2} (\N^T)^k f F_{ik \ga},\\
\N^2 f(e_{\ga},e_{\gb}) =&\ 0.
\end{align*}
\begin{proof} The result follows directly from the general formula
\begin{align*}
(\N^2 f)(e_A, e_B) =&\ f_{,AB} - \gG_{AB}^C f_{,C},
\end{align*}
and comparison against Lemma \ref{Christoffel}.
\end{proof}
\end{lemma}

\section{Invariant metrics and pluriclosed flow} \label{s:invPCF}

In this section we investigate the pluriclosed flow in the setting of invariant 
metrics on Sasaki-type complex surfaces.  First in Proposition 
\ref{p:invpreserved} we show that invariance is preserved by the flow 
equations.  Building on this we show in Proposition \ref{p:flowreduction} how 
to 
decompose the pluriclosed flow equations into a flow of a transverse metric and 
a Hermitian connection form.

\begin{prop} \label{p:invpreserved}  Let $(M^{2n}, 
g, J)$ admit a holomorphic 
Killing field $X$.  Let 
$g_t$ be the
solution to pluriclosed flow flow with this initial condition.  Then
 $X$ is a Killing field for $g_t$ for all $t \geq 0$.
\begin{proof} Since the $J$ is fixed by pluriclosed flow, $X$ certainly remains 
holomorphic.  To show $X$ remains a Killing field it is thus equivalent to show 
that $L_{X} \gw_t \equiv 0$ along the flow.  First note that as $\gw_t \in 
\Lambda^{1,1}$ is pluriclosed and $X$ is holomorphic, it follows that $L_{X} 
\gw_t \in 
\Lambda^{1,1}$, and is also pluriclosed.  We note that the linearization of 
$-\rho_B^{1,1}$ acting on pluriclosed
$(1,1)$ tensors is a linear elliptic operator 
with symbol that of the Laplacian (cf \cite{PCF} Proposition 3.1), which we 
denote $\mathcal L$.  We thus derive a heat equation for $L_X 
\gw_t$,
\begin{align*}
 \dt L_X \gw =&\ L_X (- \rho_B^{1,1})\\
 =&\ [D (- \rho_B^{1,1})] (L_X \gw)\\
 =&\ \mathcal L (L_X \gw).
\end{align*}
It follows from a standard argument that the condition $L_{X} \gw \equiv 0$ is 
preserved by uniqueness of solutions to this linear parabolic system.
\end{proof}
\end{prop}

In view of this and Proposition \ref{p:KKmetric}, we expect the 
pluriclosed flow to reduce to a flow of triples $(g^T, \mu, f)$.  Although, 
since all of the metrics are pluriclosed, by Lemma \ref{metriclemma20} we 
expect $f$ to remain constant, a fact reflected by the vanishing of the Bismut 
curvature in these directions as in Proposition \ref{STrhoB}.
To confirm this we need a preliminary lemma indicating how to vary Hermitian 
connection forms in a 
manner analogous to varying a Hermitian metric on a vector bundle and producing 
the associated Hermitian connections.

\begin{lemma} \label{l:Hermconnvar}  Given a 
Sasaki-type complex surface $(M^4, 
J)$, a Hermitian 
connection form $\mu \in \Lambda^1(\mathfrak t^2)$, and $f_1, f_2$ 
invariant functions, the form
\begin{align*}
\mu_{f} := \left( \mu^Z + d^c f_1 + d f_2 \right) \otimes Z + \left( \mu^W + 
d^c 
f_2 - d f_1 \right) \otimes W
\end{align*}
is Hermitian connection.
\begin{proof} Since the $f_i$ are $\{Z,W\}$-invariant, conditions (1) and (3) 
of 
Definition \ref{Hcl} follow.  To check condition (2) we compute
\begin{align*}
\mu_f(JX) - J_{\mathfrak t^2} \mu(X)=&\ \mu (JX) - J_{\mathfrak t^2} \mu(X) + 
\left( d^c f_1(JX) + 
d 
f_2(JX) \right) Z + \left( d^c f_2(JX) - d f_1(JX) \right) W\\
&\ - J_{\mathfrak t^2} \left( \left(d^c f_1(X) + d f_2(X) \right) Z + \left( 
d^c f_2(X) - 
d 
f_1(X) \right) W \right)\\
=&\ \left( df_1(X) - d f_1(X) + d f_2(J X) + d^c f_2(X) \right) Z\\
&\ + \left( d^c f_2(JX) - df_2(X) - d f_1(JX) - d^c f_1(X) \right) W\\
=&\ 0,
\end{align*}
since $d^cf(X) = - df(JX)$.
\end{proof}
\end{lemma}

\begin{prop} \label{p:flowreduction} Given a Sasaki-type complex surface $(M^4, 
J)$ and $\mu$ a Hermitian connection form, suppose $(g^T, 
f_1, f_2)$ is a one-parameter family of invariant transverse metrics and 
functions 
satisfying
\begin{gather} \label{f:reducedPCF}
\begin{split}
\frac{\del g^T}{\del t} =&\ - \tfrac{1}{2} \left( R^{T} - \brs{F_{\mu_f}}^2 
\right) g^T,\\
\dt f_1 =&\ \tfrac{1}{2} \tr_{\gw^T} F^{Z}_{\mu_f} = \tfrac{1}{2} \gD_{g_T} f_1 
+ \tfrac{1}{2} \tr_{\gw^T} F^Z_{\mu},\\
\dt f_2 =&\ \tfrac{1}{2} \tr_{\gw^T} F^{W}_{\mu_f} = \tfrac{1}{2} \gD_{g_T} f_2 
+ \tfrac{1}{2} \tr_{\gw^T} F^W_{\mu}.
\end{split}
\end{gather}
 Then the one-parameter of associated metrics $g_t = g(g^T, \mu_f,1)$ is a 
solution to pluriclosed flow.
\begin{proof} It suffices to show that the associated family of K\"ahler forms 
evolves by $- \rho_B^{1,1}$.  To that end we fix an adapted frame at some 
point $(p,t) \in M \times \{t\}$ and compute
\begin{align*}
 \dt \left( \gw_g \right)_{12} =&\ - \dt (g^T)_{11}\\
 =&\ \tfrac{1}{2} \left( R^T - \brs{F_{\mu}}^2 \right) g^T_{11}\\
 =&\ - (\rho_B^{1,1})_{12},
 \end{align*}
where the last line follows by comparing against Proposition \ref{STrhoB}.  
Next we compute
 \begin{align*}
  \dt \left( \gw_g \right)_{14} =&\ \dt \left( g_{13} J_4^3 \right) = - \dt 
g_{13} = - \dt (\mu_f^Z)_1\\
  =&\ - \tfrac{1}{2} (d^c \tr_{\gw^T} F^Z)_1 - \tfrac{1}{2}  (d \tr_{\gw^T} 
F^W)_1\\
  =&\ \tfrac{1}{2} (d \tr_{\gw^T} F^Z)_2 - \tfrac{1}{2} (d \tr_{\gw^T} F^W)_1\\
  =&\ \tfrac{1}{2} (\rho_B)_{23} - \tfrac{1}{2} (\rho_B)_{14}\\
  =&\ - (\rho_B^{1,1})_{14}.
 \end{align*}
Similarly we have
\begin{align*}
  \dt \left( \gw_g \right)_{13} =&\ \dt \left( g_{14} J_3^4 \right) = \dt 
g_{14} = \dt (\mu_f^W)_1\\
  =&\ \tfrac{1}{2} (d^c \tr_{\gw} F^W)_1 - \tfrac{1}{2} (d \tr_{\gw} F^Z)_1\\
  =&\ - \tfrac{1}{2} (d \tr_{\gw} F^W)_2 - \tfrac{1}{2} (d \tr_{\gw} F^Z)_1\\
  =&\ - \tfrac{1}{2} (\rho_B)_{24} - \tfrac{1}{2} (\rho_B)_{13}\\
  =&\ - (\rho_B^{1,1})_{13}.
\end{align*}
The proposition follows.
\end{proof}
\end{prop}

\begin{rmk} \label{r:RYM} 
Let us consider the standard Sasakian structure on $S^3$ generated by the Hopf 
action on $\mathbb C^2$, where the resulting complex surface is the standard 
diagonal Hopf surface, which is a principal $T^2$ fibration over $\mathbb 
{CP}^1$.  One can easily compute that the curvature tensors $F^Z$ and $F^W$ 
both evolve by the Hodge Laplacian heat flow, i.e. we have the system of 
equations on $\mathbb{CP}^1$,
\begin{align*}
\dt g =&\ - \tfrac{1}{2} R_g g + \tfrac{1}{2} \brs{F}^2_g g = - \Rc_g + F^2,\\
\dt F =&\ - \tfrac{1}{2} \gD_d F.
\end{align*}
where $F^2_{ij} = F_{ik} F_{jk}$.  These equations are, up to immaterial global 
scaling factors, a natural coupling of the Ricci and Yang-Mills flows 
introduced independently by the author \cite{Streetsthesis} and Young 
\cite{Youngthesis}.  These equations result from studying the Ricci flow of an 
invariant metric on a principal bundle, but fixing the metric on the fibers.  
While freezing the bundle metric may seem natural from a Yang-Mills point of 
view, it is arguably unnatural from the point of view of the geometry of the 
total space (cf. \cite{LottDR} for a discussion of the Ricci flow of an 
invariant 
metric on a principal bundle).  Thus it is somewhat surprising that the 
pluriclosed flow, defined on general complex manifolds with no symmetry 
considerations in mind, should naturally freeze the metric on the fibers in 
this invariant setting.
\end{rmk}

\begin{prop} \label{p:solitonreduction} An invariant pluriclosed Hermitian 
metric $g = g(g^T,\mu)$ on a Sasaki-type complex surface is a steady soliton 
with defining function $f$ if 
and only if $f$ is invariant and there exists $\gl \in \mathbb R$
such that
\begin{gather} \label{f:solitonreduction}
\begin{split}
  (R^T - \brs{F^Z}^2) g^T + L_{\N f} g^T =&\ 0,\\
e^{-f} \tr_{\gw^T} F^Z =&\ \gl_1,\\
F^W =&\ 0.
\end{split}
\end{gather}
\begin{proof} 
 To see that $f$ is invariant, we compute using Proposition \ref{STrhoB} and 
Lemma \ref{l:LeeformLDlemma}
 \begin{align*}
 0 =&\ \rho_B^{1,1}(JZ, Z) - \tfrac{1}{2} L_{\theta^{\sharp}} g (Z,Z) + 
\tfrac{1}{2} L_{\N f} g(Z,Z)\\
 =&\ \N^2 f(Z,Z)\\
 =&\ Z (Z f) - (\N_Z Z) f\\
 =&\ Z (Z f),
 \end{align*}
where the last line follows since $Z$ is a constant length Killing field, hence 
$\N_Z Z = \grad \brs{Z}^2 = 0$.  We integrate this against $e^{-f} dV_g$ to 
yield
 \begin{align*}
 0 =&\ \int_M Z (Z f) e^{-f} dV_g\\
 =&\ \int_M \left\{ Z \left( e^{-f} Zf \right) + e^{-f} (Z f)^2 \right\} dV_g\\
 =&\ \int_M L_Z \left( e^{-f} Zf dV_g \right) + \int_M (Zf)^2 e^{-f} dV_g\\
 =&\ \int_M (Z f)^2 e^{-f} dV_g.
 \end{align*}
 Thus $Z f \equiv 0$, and similarly $W f \equiv 0$.
 
Now note that as a consequence of the definition of soliton 
and Proposition \ref{p:BismutGR} we see that
\begin{align*}
0 = \Rc - \tfrac{1}{4} H^2 + \tfrac{1}{2} L_{\N f} g =&\ \rho^{1,1}_B(J \cdot, 
\cdot) - \tfrac{1}{2} L_{\theta^{\sharp}} g + \tfrac{1}{2} L_{\N f} g.
\end{align*}
 Pairing this equation against two horizontal vectors and using Proposition 
\ref{STrhoB} and Lemmas \ref{l:LeeformLDlemma} and \ref{l:Hessiandecomp} yields 
the 
first equation of (\ref{f:solitonreduction}).  Next fix a horizontal vector 
$e_i$ and note that the above equation again in conjunction with Proposition 
\ref{STrhoB} and Lemmas \ref{l:LeeformLDlemma} and \ref{l:Hessiandecomp} implies
 \begin{align*}
 0 =&\ \rho^{1,1}_B(JZ, e_i) - \tfrac{1}{2} L_{\theta^{\sharp}} g (Z,e_i) + 
\tfrac{1}{2} L_{\N f} g(Z,e_i)\\
 =&\ - \rho_B^{1,1}(e_i,W) + \tfrac{1}{2} e_i (\tr_{\gw} F^W) - \tfrac{1}{2} 
\N_k f F^Z_{ik}\\
 =&\ - \tfrac{1}{2} e_i \hook (d^c \tr_{\gw^T} F^Z) + \tfrac{1}{2} e_i \hook 
(\N f \hook F^Z),
 \end{align*}
So that
 \begin{align} \label{f:solitonFform}
0 =&\ - d^c \tr_{\gw^T} F^Z + \N f \hook F^Z.
 \end{align}
Let us expand this component wise with respect to an adapted frame to observe 
the equations
\begin{align*}
0 =&\ \left( - d^c \tr_{\gw^T} F^Z + \N f \hook F^Z \right)_1\\
=&\ e_2 \tr_{\gw^T} F^Z + e_2 f F_{21}\\
=&\ e^{f} e_2 \left( e^{-f} \tr_{\gw^T} F^Z \right),
\end{align*}
and similarly
\begin{align*}
0 =&\ \left( - d^c \tr_{\gw^T} F^Z + \N f \hook F^Z \right)_2\\
=&\ - e_1 \tr_{\gw^T} F^Z + e_1 f F_{12}\\
=&\ - e^f e_1 \left( e^{-f} \tr_{\gw^T} F^Z \right).
\end{align*}
Hence
\begin{align}
e^{-f} \tr_{\gw^T} F^Z \equiv \mbox{constant} = \gl,
\end{align}
which is the second equation of (\ref{f:solitonreduction}).

 A similar argument shows that
 $e^{-f} \tr_{\gw^T} F^W = \mbox{constant} = \gl'$. Since by construction $W$ 
is tangent to the product $S^1$ action over $S^3$, it follows that $F^W = d a$ 
for an invariant $1$-form $a$.  If $\gl' \neq 0$, then we would obtain $\gw^T = 
\frac{e^{-f}}{\gl'} d a$, contradicting that $\gw^T$ is positive definite.  
Thus $\gl' = 0$, and this implies $F^W = 0$.

\end{proof}
\end{prop}

\begin{rmk} \label{r:notSLCK} We pause here to note that the only case of a 
soliton in this ansatz which is also locally conformally K\"ahler is the 
standard metric on the diagonal Hopf surfaces.  In particular, taking the Hodge 
star of the second soliton equation yields that $d (e^{-f} \theta) = 0$.  If 
the metric was also locally conformally K\"ahler, then $d \theta = 0$, and one 
then concludes $d f \wedge \theta = 0$.  Since $f$ is basic, comparing against 
Lemma \ref{l:Leeformlemma} thus implies that $d f = 0$, and so the metric is a 
fixed point of pluriclosed flow, and thus is the standard metric on a diagonal 
Hopf surface (see \S \ref{s:background}).
\end{rmk}

\section{Existence Proofs} \label{s:existence}

In this section we complete the proof of Theorem \ref{t:solitonthm}.  First we 
observe a further a priori Killing field present in this setting.  Using this 
we provide two conceptually distinct but ultimately equivalent reductions of 
the soliton system to ordinary differential equations.  First we consider 
Hopf surfaces where the underlying Sasaki manifold is quasiregular, in which 
case the quotient space is an orbifold.  In this setting the extra Killing 
field corresponds to the natural rotational symmetry on this orbifold, and we 
reduce the soliton to equations in the underlying arclength parameter on this 
orbifold.

\subsection{A further a priori symmetry}

Thus far we have only set up the soliton and flow equations with two real 
holomorphic Killing fields on a complex surface, and thus it is a ``codimension 
2'' construction, and one would still expect to use methods of partial 
differential equations to find solutions.  However, building upon a fundamental 
observation in the theory of Ricci solitons on surfaces (\cite{Ricciflowbook} 
p. 241, \cite{ChenLuTian}), we see that 
the equations automatically acquire an extra symmetry, and thus we are in a 
``codimension 1'' situation, which can be addressed by ODE methods.

\begin{prop} \label{p:invprop} Let $(g,f)$ be an invariant soliton on a 
Sasaki-type complex surface $(M^4, J)$.  Then
\begin{align*}
 L_{J \N f} g^T = L_{J \N f} \gw^T = L_{J \N f} F^Z = 0.
\end{align*}

\begin{proof} Since the function $f$ is invariant, $\N f$ is a horizontal 
vector field.  A direct calculation shows that for an invariant horizontal 
vector field $W$ one has
\begin{align*}
 L_{W} g^T (Y,Z) = \N^T W^{\flat} (Y,Z) + \N^T W^{\flat}(Z,Y).
\end{align*}
Moreover, since $J$ is invariant, and preserves $\HH$ by Lemma 
\ref{l:Hermitianbundle}, it follows that $J \N f$ is horizontal and invariant.  
Since the 
transverse structure is K\"ahler, a short calculation shows that
\begin{align*}
 \N^T (J \N f)^{\flat} (Y,Z) + \N^T (J \N f)^{\flat}(Z,Y) = \N^T \N^T f (JY, Z) 
+ \N^T \N^T f(JZ, Y)
\end{align*}
But from the reduced solitons equations (\ref{f:solitonreduction}), we know 
that $\N^T \N^T f$ is of type $(1,1)$, hence the above quantity must vanish, as 
required.

Next, since $d \gw^T = 0$ by Lemma \ref{metriclemma20}, we see by the Cartan 
formula that
\begin{align*}
L_{J \N f} \gw^T =&\ d (J \N f \hook \gw^T) = - d (d f) = 0.
\end{align*}

To show the invariance of $F$, we first note that using the second equation of 
(\ref{f:solitonreduction})
\begin{align*}
L_{J \N f} \tr_{\gw^T} F^Z =&\ (J \N f) \hook d \tr_{\gw^T} F^Z = - \N f \hook 
d^c \tr_{\gw^T} F^Z = - \N f \hook \N f \hook F^Z = 0.
\end{align*}
Since we can 
express $F^Z = \tr_{\gw^T} F^Z \gw^T$, the invariance of $F^Z$ now follows.
\end{proof}
\end{prop}

\subsection{ODE reduction in quasiregular case} \label{ss:ODE1}

As discussed in \S  \ref{s:inducedsurfaces}, the quotient orbifold is only 
singular at cone points, of which there are no more than two.  On the smooth 
part, we note that Proposition \ref{p:invprop} implies that a hypothetical 
soliton has $J \N f$ as a Killing field.  Moreover it implies that $J \N f$ is 
holomorphic, and in fact must correspond to the natural holomorphic rotational 
symmetry present on bad orbifolds.  Such a metric can be expressed with respect 
to polar coordinates as
\begin{align} \label{f:gsym}
g = dr^2 + \phi^2(r) d \theta^2,
\end{align}
Note also that the bundle curvature $F^Z$ is invariant, and so we may express
\begin{align} \label{f:Fsym}
F^Z = \gamma(r) dr \wedge d \theta.
\end{align}
It will also be useful to work in terms of a combined quantity
\begin{align} \label{f:psidef}
\psi := \tr_{\gw} F^Z,
\end{align}
which is related to $\phi$ and $\gg$ using Lemma \ref{l:symmcurvature} below.
Also note that for a metric as in (\ref{f:gsym}) to correspond to a metric on 
an orbifold 
means that we have $\phi(0) = \phi(L) = 0$ for some $L > 0$, and moreover if 
our cone points have angles $\gb_1$ and $\gb_2$, then we require $\phi'(0) = 
\frac{\gb_1}{2\pi}, \phi'(L) = - \frac{\gb_2}{2 \pi}$.

\begin{lemma} \label{l:symmcurvature} Given the setup above, one has
\begin{enumerate}
\item $J \del_r = \phi^{-1} \del_{\theta}$, $J \del_{\theta} = - \phi \del_r$,
\item $R = - 2 \frac{\phi_{rr}}{\phi}$,
\item $\psi = \tr_{\gw^T} F^Z = \frac{\gamma}{\phi}$, \quad $\brs{F^Z}^2 = 2 
\frac{\gg^2}{\phi^2}$,
\item $d^c \tr_{\gw_T} F^Z = \left( \gamma_r - \frac{\gamma \phi_r}{\phi} 
\right) d \theta$,
\item For a radial function $f$ one has $\N f \hook F = f_r \gamma d \theta$,
\item For a radial function $f$ one has $\N^2 f = f_{rr} dr^2 + \phi \phi_r f_r 
d \theta^2$.
\end{enumerate}
\begin{proof} 
To understand the complex structure, note that by (\ref{f:gsym}) and the fact 
that $g$ is compatible with $J$ we see that $J \del_r$ must be a multiple of 
$\del_{\theta}$.  But then appealing to compatibility again we obtain
\begin{align*}
1 = g(\del_r, \del_r) = g(J \del_r, J \del_r),
\end{align*}
and so it follows that $J \del_r = \phi^{-1} \del_{\theta}$, using the standard 
orientation.  The equation $J \del_{\theta} = - \phi \del_r$ thus follows via 
$J^2 = - \Id$.  The scalar curvature in item (2) is a standard calculation we 
omit.   For item (3), note using our conventions that in these coordinates we 
have
 \begin{align*}
 \tr_{\gw^T} F^Z = F_{r \theta} g^{rr} J_r^{\theta} = \frac{\gamma}{\phi},
 \end{align*}
 as claimed.  The square norm of $F^Z$ in item (3) follows easily from 
(\ref{f:gsym}) and (\ref{f:Fsym}).  Using the general formula $d^c f(X) = - 
df(JX)$, by the rotational invariance 
it is clear that the only nonvanishing component for $d^c \tr_{\gw^T} F^Z$ is a 
multiple of $d \theta$.  Thus we observe, using (4),
\begin{align*}
d^c \tr_{\gw^T} F^Z (\del_{\theta}) = - d \tr_{\gw_T} F^Z( - \phi \del_r) = 
\phi \left( \tr_{\gw^T} F^Z \right)_r = \gamma_r - \frac{\gamma \phi_r}{\phi}
\end{align*}
For item (6) note first that for a radial function $f$ one has $d f = f_r dr$, 
and so by (\ref{f:gsym}) it follows that $\N f = f_r \del_r$.  Combining this 
with (\ref{f:Fsym}) it is clear that $\N f \hook F^Z = f_r \gamma d \theta$, as 
claimed.
\end{proof}
\end{lemma}

\begin{lemma} \label{l:symmsoliton} Let $(M^2, g)$ be a surface with rotational 
symmetry, expressed as in (\ref{f:gsym}).  Then this is a soliton if and only 
if there exists a constant $A$ such that
\begin{align*}
J \grad f =&\ A \del_{\theta}, \qquad 0 = A \phi_r - \frac{\phi_{rr}}{\phi} - 
\frac{\gg^2}{\phi^2}, \qquad \psi_r = A \phi \psi.
\end{align*}
\begin{proof} Using the metric soliton equation of (\ref{f:solitonreduction}) 
with items (1) and (7) of Lemma \ref{l:symmcurvature} we obtain
\begin{align*}
0 =&\ \tfrac{1}{2} \left( R - \brs{F}^2 \right) g + \N^2 f,\\
=&\ \left( - \frac{\phi_{rr}}{\phi} - \frac{\gamma^2}{\phi^2} \right) \left( 
dr^2 + \phi^2 d \theta^2 \right) + f_{rr} dr^2 + \phi \phi_r f_r d \theta^2.
\end{align*}
Looking at the different components we obtain the equations
\begin{gather} \label{f:symmsoliton10}
\begin{split}
0 =&\ f_{rr} - \frac{\phi_{rr}}{\phi} - \frac{\gamma^2}{\phi^2},\\
0 =&\ \frac{\phi_r f_r}{\phi} - \frac{\phi_{rr}}{\phi} - 
\frac{\gamma^2}{\phi^2}.
\end{split}
\end{gather}
Combining these two yields
\begin{align*}
(\log f_r)_r = \frac{f_{rr}}{f_r} = \frac{\phi_r}{\phi} = (\log \phi)_r
\end{align*}
Integrating this we see that $f_r = A \phi$, for an as yet undetermined 
constant $A$.  Comparing against Lemma \ref{l:symmcurvature} (1) it follows 
that $J \grad f = A \del_{\theta}$, as claimed.  Note that a formula of this 
type was inevitable from the construction, as we already knew that $J \grad f$ 
generated the rotational symmetry.  Also plugging $f_r = A \phi$ into 
(\ref{f:symmsoliton10}) yields the second equation of the lemma.  For the final 
equation we simply differentiate the equation $e^{-f} \psi = \gl$ to yield
\begin{align*}
 0 =&\ \psi_r - f_r \psi = \psi_r - A \phi \psi,
\end{align*}
as claimed.
\end{proof}
\end{lemma}

\subsection{ODE reduction in general case} \label{ss:ODEirreg}

Here we reduce the soliton equations to a system of ODE in the general case.  While the presentation is different, this reduction is a generalization of that in \S \ref{ss:ODE1}.  

Fix $\ga,\gb \in \mathbb C$, determining a Hopf surface as described in \S 
\ref{s:inducedsurfaces}.  We freely adopt the notation of that section here.  Note that, in the irregular case, 
the generic orbit of the Reeb vector $Z$ is dense in the 
standard torus $T_{\gl} := \{ \brs{z_1}^2 = \gl^2, \brs{z_2}^2 = 1 - \gl^2\}$ 
containing it, it follows that basic functions are constant on 
such tori.  As the vector field $E_1$ is tangent to these tori, it follows that 
basic functions are invariant under $E_1$ as well.  While this is not true for general invariant functions in the quasiregular case, we nonetheless impose $E_1$ invariance as an ansatz.  Thus, as $\{E_1,E_2\}$ span 
the contact distribution, we expect to reduce the soliton equation to an ODE 
along the $E_2$ direction.  In particular, let $X = \gs^{-1} E_2$, and define a 
parameter
\begin{align*}
s(t) = \frac{b}{2} \ln \left( b - t 
\right) - \frac{a}{2} \ln \left(t - a \right),
\end{align*}
defined for $t \in (a,b)$, (note that as in \S \ref{s:inducedsurfaces} we have 
$a < b$).
Observe that
\begin{align*}
X( s(\gs) ) =&\ \tfrac{1}{\gs} \left\{ \brs{z_2}^2 \left( x_1 \del_{x_1} + y_1 
\del_{y_1} \right) - \brs{z_1}^2 \left( x_2 \del_{x_2} + y_2 \del_{y_2} \right) 
\right\} \left\{ - \frac{a}{2} \ln \left(\gs - a \right) + \frac{b}{2} \ln 
\left( b - \gs  \right) \right\}\\
=&\ \tfrac{1}{\gs} \left\{ \brs{z_2}^2 \left( - \frac{a}{\gs - a} \left( a 
x_1^2 + a y_1^2\right) + \frac{b}{b - \gs} \left( -a x_1^2 - a y_1^2 \right) 
\right) \right.\\
&\ \left. \qquad - \brs{z_1}^2 \left( - \frac{a}{\gs - a} \left( b x_2^2 + b 
y_2^2\right) + \frac{b}{b - \gs} \left( -b x_2^2 - b y_2^2 \right) \right) 
\right\}\\
=&\ \tfrac{\brs{z_1}^2 \brs{z_2}^2}{\gs} \left\{ - \frac{a^2}{\gs - a} - 
\frac{ab}{b - \gs} + \frac{ab}{\gs - a} + \frac{b^2}{b - \gs} \right\}\\
=&\ 1,
\end{align*}
where the last line follows by judiciously applying the identities $\gs - a = 
(b - a)\brs{z_2}^2, b - \gs = (b-a) \brs{z_1}^2$.

We now define functions which describe the metric and bundle curvature as in \S 
\ref{ss:ODE1}.  First, for a given invariant metric, define $\phi$ and $\xi 
\geq 0$ via
\begin{align*}
\phi^2 = \xi := \tfrac{1}{2} \gw^T(E, J \bar{E}).
\end{align*}
Also, we set
\begin{align*}
\psi = \tr_{\gw^T} F^Z.
\end{align*}
Next we must determine the boundary conditions, i.e. the behavior at the points 
$z_2 = 0$, or $z_1 = 0$, corresponding to $s \to \infty, s \to -\infty$, or 
$\gs = a, \gs = b$, respectively.  First, as a function of $\gs$, it follows 
from the definition of $\xi$ and the nondegeneracy of the metric that $\xi$ 
must be of order $\brs{z_2}^2$ near $z_2 = 0$.  To convert this to the 
parameter $s$, we observe the formula
\begin{align*}
 \left(b - a \right) \brs{z_2}^2 = e^{-\frac{2}{a} s} \left( (b - a) 
\brs{z_1}^2 \right)^{\frac{b}{a}}
\end{align*}
Hence, expanded as a power series in $e^{-s}$, the leading order term is 
proportional to $e^{-\frac{2}{a} s}$.  It follows easily that the function 
$\phi$ has such an expansion with leading order term $e^{-\frac{1}{a} s}$.  It 
follows easily that the $Y$ derivative of $\phi$ at this boundary point is 
$-\frac{1}{a}$.  An identical argument shows that the $Y$ derivative of $\phi$ 
at the boundary point $z_1 = 0$ is $\frac{1}{b}$, as claimed.

\begin{lemma} \label{l:irreglem10} Let $(M^4, J)$ be an irregular Sasaki-type 
complex surface.  An 
invariant metric $(g^T, \mu)$ determines a steady pluriclosed soliton if and 
only if there exists a constant $A$ such that
\begin{gather} \label{l:irregODE}
\begin{split}
0 = A \phi_r - \frac{\phi_{rr}}{\phi} - 
\psi^2, \qquad \psi_r = - A \phi \psi,
\end{split}
\end{gather}
where $\tfrac{\del}{\del r} = \frac{1}{\phi \gs} E_2$.
\begin{proof} 
Considering the transverse piece of the reduced soliton equations in 
Proposition \ref{p:solitonreduction},
we obtain the fact that the transverse Hessian of $f$ is pure trace, i.e.
\begin{align} \label{f:ired10}
\N^T \N^T f = \tfrac{1}{2} \gD^T f g^T.
\end{align}
We will use this condition to first of all determine an explicit relationship 
between $f$ and $g^T$.  The calculations can be effectively globalized using 
the frame $\{E_1, E_2\}$.  Let
\begin{align*}
E = \gs^{-1} \left(E_1 - \i E_2 \right),
\end{align*}
and note that $E$ spans  the space of transverse $(1,0)$-vector fields 
everywhere 
except the points where $z_1 = 0$ or $z_2 = 0$.  Since (\ref{f:ired10}) implies 
that the $(2,0) + (0,2)$ piece of the transverse 
Hessian of $f$ vanishes, it follows that
\begin{align*}
E E f - E (\log \xi) E f = 0.
\end{align*}
Since basic functions are also invariant under $E_1$ as described above, we see 
that this implies
\begin{align*}
- \gs^{-1} E_2 \left( \gs^{-1} E_2 f \right) + \gs^{-1} (E_2 \log \xi) \left( 
\gs^{-1} E_2 f \right),
\end{align*}
hence
\begin{align*}
E_2 \log \gs^{-1} E_2 f = E_2 \log \xi,
\end{align*}
and thus there exists a constant $A$ such that
\begin{align} \label{f:ired15}
X f = \gs^{-1} E_2 f = A \xi.
\end{align}
Let $\til{X} = \pi_H X$, the horizontal projection, and then note that, since 
$f$ is basic,
\begin{align} \label{f:ired20}
g^T(\N f, \til{X}) = X f = A \xi,
\end{align}
and hence
\begin{align} \label{f:ired25}
\N f = A \til{X}.
\end{align}
Also, using the $E_1$ invariance of $\xi$, we obtain the formula for the 
transverse scalar curvature
curvature 
\begin{align*}
R^T = - \xi^{-1} X X \log \xi.
\end{align*}
Combining these observations we finally obtain that the transverse piece of the 
soliton equation reduces to
\begin{align} \label{f:ired30}
0 =&\ A X \xi - X^2 \log \xi - \psi^2 \xi.
\end{align}
To connect this to the point of view in \S \ref{ss:ODE1}, we need to choose an 
arclength 
parameter for $X$.  In particular, recall that $\phi^2 = \xi$, and 
let $\frac{\del}{\del r} = \phi^{-1} X = \frac{1}{\phi \gs} E_2$ as in the 
statement.  Observe then that (\ref{f:ired30}) implies
\begin{align*}
 0 =&\ A \left(\phi \frac{\del}{\del r} \right) \phi^2 - 2 \left(\phi 
\frac{\del}{\del r} \right) \left(\phi \frac{\del}{\del r} \right) \log \phi - 
\psi^2 \phi^2\\
 =&\ 2 A \phi^2 \phi_r - 2 \phi \phi_{rr} - \psi^2 \phi^2,
\end{align*}
from which the first claimed equation follows by dividing by $\phi^{-2}$.
Also, it is elementary to differentiate the equation $e^{-f} \psi = \gl_1$ with 
respect to $X$ and apply (\ref{f:ired15}) to obtain
\begin{align*}
0 =&\ X (e^{-f} \psi) = e^{-f} \left( X \psi - A \xi \psi \right),
\end{align*}
which in turn directly implies
\begin{align*}
 \psi_r =&\ \phi^{-1} X \psi =  A \phi^{-1} \xi \psi = A \phi \psi.
\end{align*}
The lemma follows.
\end{proof}
\end{lemma}

\subsection{Solutions}

In this subsection, we construct solutions to the system of ODEs derived in 
Lemma \ref{l:irreglem10}.
First let us address the constant $A$.  Observe that if $A = 0$, it follows 
directly that $\psi_r = 0$, and 
so after scaling we can assume $\psi \equiv 1$, and we obtain the ODE 
$\phi_{rr} = - \phi$.  Thus the metric on the only 
possible solution then corresponds to the round metric $S^2$.  Indeed the 
resulting 
metric corresponds to the Hopf metric (\ref{f:hopfmetric}).

Thus, now assuming $A \neq 0$, we perform a change of variables which greatly 
simplifies the 
system and causes the parameter $A$ to drop out.  In particular, let
\begin{align*}
x = A \phi, \qquad y = A \phi_r, \qquad z = A^{\tfrac{1}{2}} 
\psi.
\end{align*}
Then from the system of ODEs one derives
\begin{gather} \label{f:ODE}
\begin{split}
x_r =&\ A \phi_r = y\\
y_r =&\ A \phi_{rr} = A \left( A \phi_r \phi - \frac{\gg^2}{\phi^2} \right) = 
xy - z^2\\
z_r =&\ A^{\tfrac{1}{2}} \psi_r = - AA^{\tfrac{1}{2}} \left( \frac{\gg \phi_r - 
A \gamma \phi^2}{\phi^2} - 
\frac{\gg \phi_r}{\phi^2} \right) = xz.
\end{split}
\end{gather}

\begin{prop} \label{p:ODEexistence} For every $1 < \rho < \infty$ there exists 
$z_0$ such that the solution to (\ref{f:ODE}) with initial condition 
$(0,1,z_0)$ exists (at least) on a finite time interval $[0,T]$ and satisfies
\begin{align} \label{f:ODEprops}
x_{| [0,T]} \geq 0, \qquad x(T) = 0, \qquad y(T) < 0, \qquad y(0) / \brs{y(T)} 
= \rho.
\end{align}
\begin{proof} The overall argument consists of finding choices of $z_0$ which 
give the required behavior first for $\rho$ close to $1$, then for $\rho$ 
large, then arguing by a continuity method that one obtains all values in 
between.

We first describe solutions with $\rho$ close to $1$.  In particular, we claim 
that for $z_0$ sufficiently large the following inequalities are preserved:
\begin{enumerate}
\item $x(t) \leq t - \tfrac{1}{2} (1 + \gd) z_0^2 t^2$,
\item $1 - (1+\gd) z_0^2 t \leq y(t) \leq 1 - (1 - \gd) z_0^2 t$,
\item $z(t) \leq (1 + \gd) z_0$
\end{enumerate}
These are certainly satisfied at time $t = 0$, so it remains to show that they 
are preserved up to the first time $T$ such that $x(T) = 0$.  First note that, 
as long as inequality (2) is preserved we see
\begin{align*}
x(t) =&\ x(0) + \int_0^t y(s) ds \leq \int_0^t \left\{ 1 - (1 - \gd) z_0^2 s 
\right\} ds \leq t - \tfrac{1}{2} (1 - \gd) z_0^2 t^2.
\end{align*}
Thus, the maximal time we must consider satisfies $t \leq \frac{2}{(1 - \gd) 
z_0^2}$, and this also shows that condition (1) is preserved as long as 
condition (2) is.  This also implies that that as long as condition (1) is 
preserved, we have the overall upper bound
\begin{align*}
x(t) \leq \sup \left(t - \tfrac{1}{2} (1 - \gd) z_0^2 t^2 \right) = \tfrac{1}{2 
(1 - \gd) z_0^2}.
\end{align*}
Thus we can integrate and estimate the differential equation for $z$ to obtain
\begin{align*}
z(t) =&\ z_0 \exp \left(\int_0^t x(s) ds \right) \leq z_0 \exp \left( 
\frac{t}{2(1 + \gd) z_0^2} \right) \leq z_0 \exp \left( \frac{1}{(1+\gd)^2 
z_0^4} \right) < (1 + \gd) z_0.
\end{align*}
for $z_0$ chosen sufficiently large.  Lastly we note that using all the 
estimates in play and integrating
\begin{align*}
y(t) \leq&\ y(0) + \int_0^t \left( x y - z^2 \right) ds\\
\leq&\ 1 + \int_0^t \left(\frac{1}{2(1+\gd) z_0^2} \left(1 - (1-\gd) z_0^2 s 
\right) - z_0^2 \right) ds\\
\leq&\ 1 + \left( \frac{1}{2(1+\gd) z_0^2} - z_0^2 \right) t\\
\leq&\ 1 - \left( 1 - \gd \right) z_0^2 t,
\end{align*}
for $z_0$ chosen sufficiently large.  A very similar integration yields the 
lower bound as well.  We have shown that conditions (1), (2), and (3) hold 
until $x(T) = 0$.  We claim that $y(0) / \brs{y(T)}$ approaches $1$ for $z_0$ 
chosen large.  To do this we first obtain a lower bound for the first time $T$ 
that $x(T) = 0$.  In particular, integrating and estimating we obtain
\begin{align*}
x(t) \geq x_0 + \int_0^t \left( 1 - (1+\gd) z_0^2 s \right) ds = t - 
\tfrac{1}{2} (1 + \gd) z_0^2 t^2
\end{align*}
Thus one sees that $T \geq \frac{2}{(1+\gd) z_0^2}$.  Returning to estimate 
(2) we thus obtain
\begin{align*}
1 - 2 \frac{1+\gd}{1-\gd} \leq y(T) \leq 1 - 2 \frac{1-\gd}{1+\gd}.
\end{align*}
Thus certainly for $\gd$ chosen sufficiently small $y(T)$ approaches $-1$, as 
claimed.

We now describe solutions corresponding to large values of $\rho$.  This is 
more involved, requiring describing three phases of the solution which we name 
the ``growth phase,'' ``control phase,'' and ``decay phase.''  By the ``growth 
phase'' we mean that, given $\gL > 0$, we can choose $z_0$ sufficiently small 
that there exists a time $t_0 > 0$ where $x(t_0) \geq \gL$.  Fix a small 
constant $\gd > 0$, and note that, as long as $z \leq \gd$, and $y \geq 0$ we 
can estimate
\begin{align*}
y(t) \geq y_0 + \int_0^t \left( xy - z^2 \right) ds \geq 1 - \gd^2 t,
\end{align*}
and so in particular for $\gd$ small we have $\inf_{[0,1]} y \geq 
\tfrac{1}{2}$, and hence $x(1) \geq \tfrac{1}{2}$.  Note then that for times $t 
\geq 1$, assuming still $z \leq \gd$ sufficiently small, we obtain the 
elementary estimate $y_t \geq \tfrac{1}{4}$, which will be preserved, and hence 
we conclude $y(t) \geq y(1) = \tfrac{1}{2}$, and thus $x(t) \geq x(1) + 
\tfrac{1}{2} t$, and so there exists a first time $t_1 \leq 2 \gL$ such that  
$x(t_1) = \gL$.  It remains to ensure we can choose $z_0$ sufficiently small to 
guarantee the hypothesis $z \leq \gd$ on a time interval of this length.  To 
that end we integrate the equation for $z$ and estimate on the time interval 
$[0,t_1]$,
\begin{align*}
z \leq&\ z_0 \exp \left( \int_0^t x(s) ds \right) \leq z_0 e^{t \gL} \leq z_0 
e^{2 \gL^2} < \gd,
\end{align*}
provided $z_0 < \gd e^{-2 \gL^2}$.

Next we have the ``control phase.''  In particular, we establish that $x$ does 
not grow without bound, but rather achieves a unique maximum value.  
Specifically, we claim that there exists a time $t_1$ such that $y(t_1) = 0$.  
To see this first note that
\begin{align*}
\frac{d}{dt} \frac{y}{z} = \frac{xy - z^2}{z} - \frac{y xz}{z^2} = - z < - z_0.
\end{align*}
Hence by an elementary integration we obtain
\begin{align*}
y(t) < z(t) \left( \frac{y_0}{z_0} - z_0 t \right) \leq 0
\end{align*}
for $t \geq \frac{y_0}{z_0^2}$, as claimed.

Lastly we have the ``decay phase,'' wherein we show $x$ returns to zero, and 
moreover that $y$ becomes very large and negative at that time.  Note that 
$y(t) \leq 0$ is certainly preserved by the ODE, and in fact for $x(t) \geq 0, 
y(t) \leq 0$ one has $y_t \leq - z^2 \leq - z_0^2$, it follows easily that 
there exists a first time $t_3$ such that $x(t_3) = 0$.  We furthermore claim 
that one has $y(t_3) \leq - \tfrac{1}{2} x(t_2)^2$.  We obtain this again via 
comparison with the idealized flow lines, in other words, we know that
\begin{align*}
y_t = xy - z^2 \leq xy = \tfrac{1}{2} (x^2)_t.
\end{align*}
Integrating the ODE $y_t = \tfrac{1}{2} (x^2)_t$ yields
\begin{align*}
y(t) = \tfrac{1}{2} x^2 + C,
\end{align*}
for some constant of integration $C$.  The flow lines are thus parabolas in a 
standard phase space diagram.  Choosing $C = - \tfrac{1}{2} x(t_2)^2$, we 
obtain the ideal boundary indicated in Figure \ref{fig:ODE}, which intersects 
the $y$-axis at the point $(0,- \tfrac{1}{2} x(t_2)^2)$.  By comparison we know 
that the solution to our ODE must lie below this curve, and thus at the time 
$t_3$ where $x(t_3) = 0$, it follows immediately that $y(t_3) \leq - 
\tfrac{1}{2} x(t_2)^2$, as claimed.
\end{proof}
\end{prop}

\begin{figure}[ht] \label{fig:ODE}
\begin{tikzpicture}[scale=0.6]

\draw [thick,->] (0,0) -- (16,0) node[right] {$x$};
\draw [thick,<->] (0,-18) -- (0,10) node[above] {$y$};

\node [draw,text width=1.27in] (v0) at (-3.8,5) {$x_0 = 0$, $y_0=1$\\ 
\small{\textcolor{red}{Red} curves: $z_0 \waymore 1$\\ \textcolor{green}{Green} 
curve: $z_0 \ll 1$}};

\draw [thick,->](v0) -- (-0.3,2.3);

\draw[thick,scale=2,domain=-1:1,smooth,variable=\y,red]  plot ({-\y*\y/4 + 
1/4},{\y});
\draw[thick,scale=2,domain=-1:1,smooth,variable=\y,red]  plot ({-\y*\y/2 + 
1/2},{\y});

\draw[thick,scale=2,domain=0:4,smooth,variable=\x,green] plot ({\x},{\x*\x/8 + 
1});
\draw[color=green,thick] (8,6) .. controls (9.9636,7.7745) and (15,7.6339) .. 
(15,0);
\draw[color=green,thick] (15,0) .. controls (15,-10.0914) and (8.7718,-13.9378) 
.. (0.0,-16.579);

\draw[thick,scale=2,domain=0:4,smooth,variable=\x,blue] plot ({\x},{\x*\x/4 + 
1});

\node [draw,text width=1.25in] (v1) at (8.5,2.3214) {\small{1. Growth phase: 
$x(t_1) \waymore 1$, $z(t_1) \ll 1$}};
\node [draw,text width=1.1in] (v2) at (8,-3) {\small{2. Control phase: $x$ 
reaches a maximum, $y(t_2) = 0$}};
\node [draw,text width=1.3in] (v3) at (4.1707,-7.8619) {\small{3. Decay phase: 
solution lies below blue curve, $x(t_3)=0$, $y(t_3) \ll 1$}};

\node(v5) at (7.0831,6.1175) {$t_1$};
\draw [fill=black] (v5) ++ (0.5,-0.5) circle (0.1);

\draw [thick,->](v1) -- (7.65,5.3333);

\node(v6) at (14.5,0.5) {$t_2$};
\draw [fill=black] (15,0)  circle (0.1);

\draw[thick,->](v2) -- (14.5,-0.3);

\draw [fill=black] (0,-16.579)  circle (0.1);
\node(v6) at (0.5,-17) {$t_3$};
\draw[thick,->](v3) -- (0.3,-16);

\draw[thick,scale=2,domain=0:{56^0.5},smooth,variable=\x,blue] plot 
({\x},{\x*\x/8 - 7});

\draw [fill=black] (0,2)  circle (0.1);
\draw [fill=black] (0,-2)  circle (0.1);

\node [draw,text width=0.8in](v7) at (-4,-4) {For $z_0 \waymore 1$, $y(T) 
\approx -1$};
\draw [thick,->](v7) -- (-0.3,-2.3);

\end{tikzpicture}
\caption{Solutions of reduced ODE system}
\end{figure}
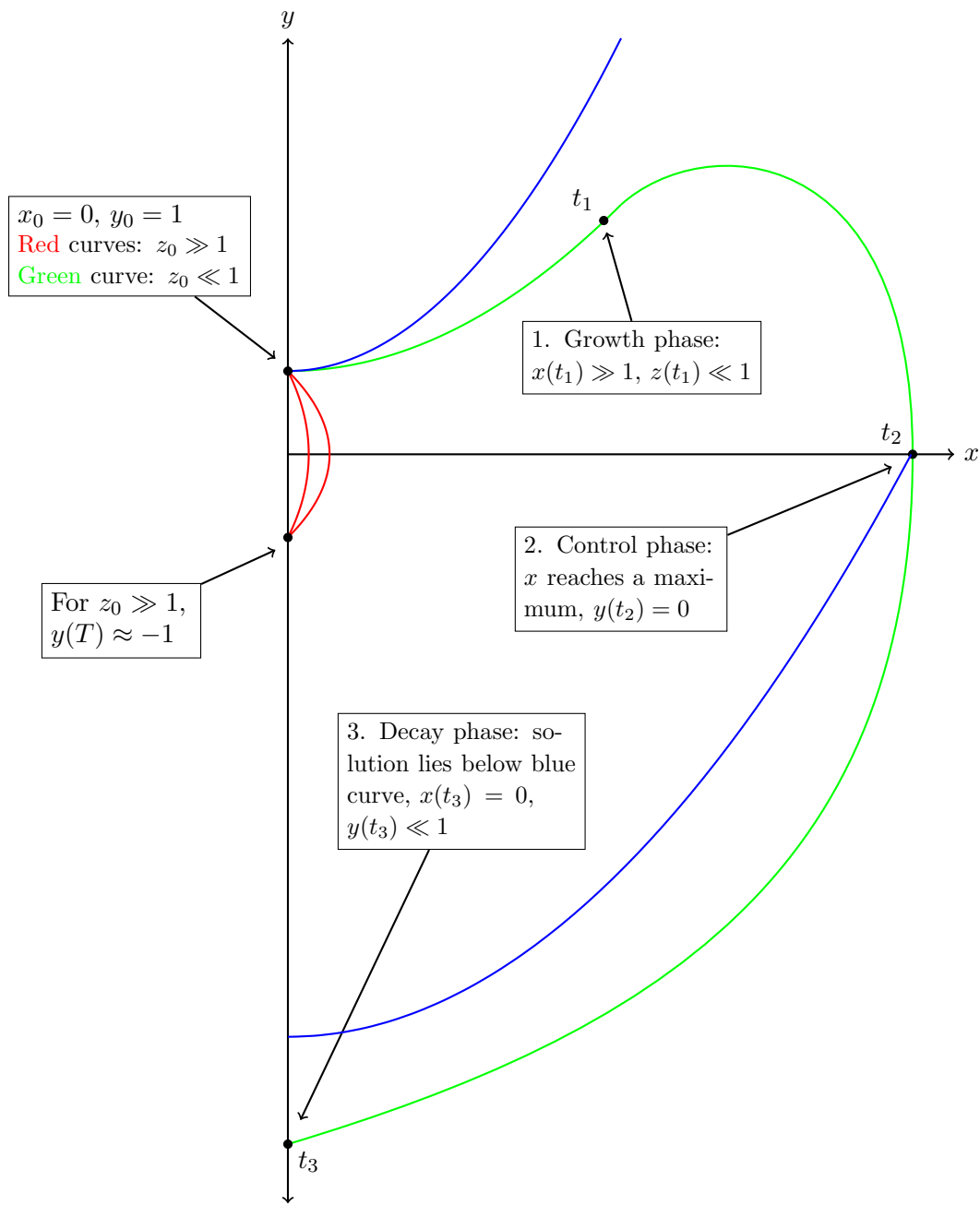

\subsection{Main proofs}

\begin{proof} [Proof of Theorem \ref{t:solitonthm}]

Fix $\ga,\gb$ $\brs{\ga} < \brs{\gb}$ determining a primary Hopf surface as in 
\S \ref{s:inducedsurfaces}.  Let $\rho = \frac{b}{a} = \frac{\ln \brs{\gb}}{\ln 
\brs{\ga}} > 1$, and choose $z_0$ and $(x(r),y(r),z(r))$ according to 
Proposition \ref{p:ODEexistence}.  By rescaling this solution as in 
(\ref{f:ODE}), we can obtain $x'(0) = \frac{1}{b}$, which by construction 
forces and $x'(T) = - \frac{1}{a}$.  Thus, as explained
in \S \ref{ss:ODEirreg}, the function $\phi = x$ defines a transverse metric 
$g^T$, with transverse K\"ahler form $\gw^T$.  This in turn defines a curvature 
form $F^Z = z \gw^T$.  Furthermore, by construction the soliton function $f$ 
will only depend on the parameter $\gs$, and is solved for using $\phi$ via 
(\ref{f:ired15}).  By construction the triple $(g^T, F^Z, f)$ solves the system 
of equations (\ref{f:solitonreduction}).
To finish we must ensure that $F^Z$
arises as the curvature of a Hermitian connection form.  To that end we 
first note that for a solution to the reduced soliton equations $(g,F^Z, f)$, 
given $\gl > 0$ one has that $(\gl 
g, \gl^{\tfrac{1}{2}} F^Z, f)$, for any $\gl > 0$.  Thus without loss of 
generality we can rescale so that $[F^Z]_B = [F^Z_{\mu_0}]$, where
$\mu_0$ denotes the connection for associated to any background 
invariant 
metric on $M$, for instance the one arising from the original Sasakian 
structure.  By the 
$\del_b\delb_b$-lemma \cite{ElKacimi} there exists an invariant
function $\zeta$ such that $F^Z = F^Z_{\mu_0} + \i \del_b \delb_b \zeta$.  It 
follows that $F^Z = 
F^Z_{\mu_{\zeta}}$, in the notation of Lemma \ref{l:Hermconnvar}.  Moreover, it 
is clear by construction that $F^W_{\mu_{\zeta}} = 0$.  The triple $(g^T, 
\mu_{\zeta}, f)$ is the claimed soliton.

Finally, we address the case of secondary Hopf surfaces.  As explained in the 
work of Kato \cite{KatoHopf, KatoHopferrata}, for Hopf surfaces of class $1$ 
with $\brs{\ga} \neq \brs{\gb}$, the fundamental group $\gG$ of $M$ is 
expressed as a semidirect product $\gG = \IP{\gg_{\ga,\gb}} \ltimes H$, where 
$H\subset U(1) \times U(1)$, the group of diagonal unitary matrices acting in 
the standard way on $\mathbb C^2$, and so it suffices to show that the solitons 
we have constructed are invariant under this torus.  As explained in 
\ref{ss:ODEirreg}, the functions $\phi$ and $\psi$ are constant on these 
tori.
Since the final metric is determined by 
these functions, natural operators, and $J$, it follows that holomorphic vector 
fields tangent to these orbits are Killing, and thus one obtains $U(1) \times 
U(1)$ invariance.
\end{proof}

\begin{proof}[Proof of Corollary \ref{c:solitoncor}] Theorem \ref{t:solitonthm} 
yields nontrivial steady soliton structures on $S^3 \times S^1$.  By taking 
products with flat tori we obtain nontrivial soliton structures on $S^3 \times 
T^k$ for all $k \geq 1$.    To obtain nontrivial solitons in dimension $n=3$, 
we note that it follows from the reduced soliton equations of Proposition 
\ref{p:solitonreduction} and elementary calculations using Lemmas 
\ref{l:Leeformlemma} and \ref{l:LeeformLDlemma} that the $1$-form $e^{-f} 
\theta$ is closed, and moreover satisfies $L_{e^{-f} \theta^{\sharp}} g = 0$, 
and thus $e^{-f} \theta$ is parallel.  Therefore the universal cover is 
isometric to a product $(S^3 \times 
\mathbb R, g' \oplus dt^2)$.  Since $H = \star \theta = \star e^f dt$, it 
follows that $\frac{\del}{\del t} \hook H = 0$.  Thus, setting $H' = i^* H$, 
where $i$ denotes the inclusion map of an $S^3$ leaf, it follows easily that 
$(H')^2 = i^*(H^2)$.  Also, in the construction of Theorem \ref{t:solitonthm}, 
we noted that the function $f$ was $Z,W$-invariant.  It follows, setting $f' = 
f \circ i$, that $(\N^2)' f' = i^* (\N^2 f)$.  Thus we conclude that for the 
structure $(g',H',f')$ on the $S^3$ leaf,
\begin{align*}
\Rc_{g'} - \tfrac{1}{4}(H')^2 + (\N^2)' f' = i^* \left( \Rc_g - \tfrac{1}{4} 
H^2 + \N^2 f \right) = 0,
\end{align*}
as required.
\end{proof}

\bibliographystyle{abbrv}

\end{document}